\documentclass[12pt]{amsart}
\usepackage{fullpage,verbatim,amssymb}
\usepackage{hyperref}
\usepackage{enumerate}

\makeatletter
\let\@@pmod\pmod
\DeclareRobustCommand{\pmod}{\@ifstar\@pmods\@@pmod}
\def\@pmods#1{\mkern4mu({\operator@font mod}\mkern 6mu#1)}
\makeatother

\newcommand{\bsl}{{\backslash}}
\newcommand{\N}{\mathbb{N}}
\newcommand{\Z}{\mathbb{Z}}

\newcommand{\R}{\mathbb{R}}
\newcommand{\TT}{\mathbb{T}}
\newcommand{\C}{\mathbb{C}}

\newcommand{\sm}{\left(\begin{smallmatrix}}
\newcommand{\esm}{\end{smallmatrix}\right)}
\newcommand{\bpm}{\begin{pmatrix}}
\newcommand{\ebpm}{\end{pmatrix}}

\newcommand{\beq}{\begin{equation}}
\newcommand{\eeq}{\end{equation}}

\DeclareMathOperator{\SL}{SL}

\DeclareMathOperator{\ASL}{ASL}

\DeclareMathOperator{\Cb}{C}

\newtheorem{theorem}{Theorem}
\newtheorem{lemma}[theorem]{Lemma}
\newtheorem{proposition}[theorem]{Proposition}

\theoremstyle{remark}

\numberwithin{theorem}{section}
\numberwithin{equation}{section}

\def\vecm{{\text{\boldmath$m$}}}

\def\vecp{{\text{\boldmath$p$}}}
\def\vecr{{\text{\boldmath$r$}}}

\def\vecs{{\text{\boldmath$s$}}}

\def\vecx{{\text{\boldmath$x$}}}

\def\vecxi{{\text{\boldmath$\xi$}}}

\def\vecnull{{\text{\boldmath$0$}}}

\def\e{\mathrm{e}}

\def\GamG{\Gamma\backslash G}

\def\scrQ{{\mathcal Q}}

\def\scrR{{\mathcal R}}

\def\e{\mathrm{e}}

\def\C{\operatorname{C{}}}

\def\SL{\operatorname{SL}}
\def\ASL{\operatorname{ASL}}

\def\ord{\operatorname{ord}}

\def\GamG{\Gamma\backslash G}

\def\trans{\,^\mathrm{t}\!}

\def\bs{\backslash}

\title{Effective equidistribution of rational points on expanding horospheres}
\author{Min Lee and Jens Marklof}
\date{\today}
\address{School of Mathematics, University of Bristol, Bristol BS8 1TW, U.K.\newline \rule[0ex]{0ex}{0ex} \hspace{8pt}{\tt min.lee@bristol.ac.uk, j.marklof@bristol.ac.uk}}

\thanks{The research leading to these results has received funding from the European Research Council under the European Union's Seventh Framework Programme (FP/2007-2013) / ERC Grant Agreement n. 291147. 
The research of the first author has been supported by EPSRC Grant {\tt EP/K034383/1}.}

\begin{document}

\begin{abstract}
Einsiedler, Mozes, Shah and Shapira [Compos.\ Math.\ 152 (2016), 667--692] prove an equidistribution theorem for  rational points on expanding horospheres in the space of $d$-dimensional Euclidean lattices, with $d\geq 3$. Their proof exploits measure classification results, but provides no insight into the rate of convergence. We pursue here an alternative approach, based on harmonic analysis and Weil's bound for Kloosterman sums, which in dimension $d=3$ yields an effective estimate on the rate of convergence.
\end{abstract}

\maketitle

\section{Introduction}
Let $d\geq 2$, $G=\SL_d(\R)$ and $\Gamma=\SL_d(\Z)$. $G$ acts by right multiplication on the quotient space $\GamG$, which carries a unique $G$-invariant probability measure $\mu$. The latter is the normalized projection of Haar measure of $G$ to $\GamG$. Set
$$
\Phi^t:=\bpm \e^{t} 1_{d-1} & \vecnull \\ \trans\vecnull & \e^{-(d-1)t} \ebpm .
$$
The {\em expanding horospherical subgroup} $H_+$ of $G$ with respect to the semigroup $\{ \Phi^t : t> 0\}$ is defined as the set of all $g\in G$ such that $\lim_{t\to\infty} \Phi^{t} g\Phi^{-t} =1_d$. 
We have explicitly
$$
	H_+ = \left\{n _+(\vecx):=\bpm 1_{d-1} & \vecnull \\ \trans\vecx & 1\ebpm \;:\; \vecx \in \R^{d-1}\right\}
$$
The corresponding {\em contracting} horospherical sugbroup $H_-$ comprises the transpose of the elements of $H_+$. It is well known that translates of patches of expanding horospheres under $\Phi^t$ become uniformly distributed in $\GamG$ with respect to $\mu$, as $t\to\infty$. We have the following equidistribution theorem.

\begin{theorem}\label{thm1}
Let $g_0\in\GamG$, $f:\GamG\times\R^{d-1}\to \R$ be bounded continuous and $\lambda$ a Borel probability measure on $\R^{d-1}$ which is absolutely continuous with respect to the Lebesgue measure. Then
$$
\lim_{t\to\infty} \int_{\R^{d-1}} f(g_0 n_+(\vecx)\Phi^t,\vecx) \, d\lambda(\vecx) = \int_{\GamG\times\R^{d-1}} f(g,\vecx) \, d\mu(g)\, d\lambda(\vecx).
$$
\end{theorem}

This theorem follows from the mixing property via Margulis' trick, and the rate of convergence can be effectively controlled \cite[Section 3.3]{Li15}. In the special case $g_0=\Gamma$ the horospheres $\{ \Gamma n_+(\vecx)\Phi^t : \vecx\in\R^{d-1}\}$ are closed, since $\{ n_+(\vecm):\vecm\in\Z^{d-1}\}\subset\Gamma$. In this case, we can replace in Theorem \ref{thm1} $\R^{d-1}$ by $\TT^{d-1}=\R^{d-1}/\Z^{d-1}$ throughout.

In the present paper we will study the case when the average with respect to $\lambda$ is replaced by an average over the rational points with denominator $q$,
\begin{equation}
\scrR_q := \{ q^{-1} \vecp : \vecp\in\Z^{d-1}\cap(0,q]^{d-1},\; \gcd(\vecp,q)=1 \} \subset (0,1]^{d-1}.
\end{equation}
In this case, for every $r\in\scrR_q$ we have by \cite[Eq.~(3.52)]{Mar10}
\beq\label{ins}
\Gamma n_+(\vecr) D(q) \in \Gamma\bs\Gamma H ,
\eeq
where 
$$
D(q):=\bpm q^{\frac{1}{d-1}} 1_{d-1} & \vecnull \\ \trans\vecnull & q^{-1} \ebpm =\Phi^{(d-1)\log q}
$$
and the subgroup
$$
H = \left\{\bpm A & \vecxi \\ 0 & 1\ebpm \;:\; A\in \SL_{d-1}(\R), \;\vecxi \in \R^{d-1}\right\} .
$$ 
For example in the case $d=3$ (which will be our focus) we have explicitly
$$
n_+(\vecr)D(q) = \bpm 1 & 0 & 0 \\ 0 & 1 & 0 \\ r_1 & r_2 & 1\ebpm \bpm \sqrt{q}  & 0 & 0 \\ 0 & \sqrt q & 0 \\ 0 & 0 & q^{-1}\ebpm.
$$

One can identify $H$ with the semi-direct product group $\ASL_{d-1}(\R):=\SL_{d-1}(\R)\ltimes\R^{d-1}$ via the group isomorphism 
$$
	\iota: \ASL_{d-1}(\R)\to H, \qquad
	(A,\vecxi) \mapsto \bpm A & \vecxi \\ 0 & 1\ebpm,
$$
where the multiplication law of $\ASL_{d-1}(\R)$ is
%
%
%
%
$$
	(A_1, \vecxi_1) (A_2, \vecxi_2)
	=
	(A_1A_2, \vecxi_1+A_1\vecxi_2). 
$$
The inclusion \eqref{ins} implies that the points $\{ \Gamma n_+(\vecr) D(q) : \vecr\in \scrR_q\}$ cannot equidistribute on $\GamG$ as $q\to\infty$. However, since 
$\Gamma \cap H \simeq \ASL_{d-1}(\Z)$
is a lattice in $H\simeq\ASL_{d-1}(\R)$, the coset $\Gamma\bs\Gamma H$ is a homogeneous space isomorphic to $\ASL_{d-1}(\Z)\bs\ASL_{d-1}(\R)$. Denote by $\mu_0$ the unique $H$-invariant probability measure on $\Gamma\bs\Gamma H$ (which is the normalized projection of Haar measure of $H$).
Einsiedler, Mozes, Shah and Shapira \cite{EMSS16} proved the following remarkable equidistribution theorem.

\begin{theorem}[\cite{EMSS16}]\label{thm2}
Let $f:\Gamma\bs\Gamma H \times \TT^{d-1}\to \R$ be bounded continuous. Then
$$
\lim_{q\to\infty} \frac{1}{\#\scrR_q} \sum_{\vecr\in\scrR_q} f(\Gamma n_+(\vecr)D(q),\vecr)  = \int_{\Gamma\bs\Gamma H\times\TT^{d-1}} f(g,\vecx) \, d\mu_0(g) \, d\vecx.
$$
\end{theorem}

This theorem has important applications to the asymptotic distribution of Frobenius numbers \cite{Mar10} and the diameters of random circulant graphs \cite{MS13} (see also the extension to Cayley graphs of general finite abelian groups \cite{SZ16}). Theorem \ref{thm2} extends the equidistribution results in \cite{Mar10} which required an additional average over $q$. 
The proof of Theorem \ref{thm2} requires deep ergodic-theoretic tools, including Ratner's measure classification theorem. 
The present work provides a different proof of Theorem \ref{thm2} in the case $d=3$, which uses harmonic analysis on $\ASL_2(\Z)\bs\ASL_2(\R)$ and Weil bounds on Kloosterman sums. Unlike the ergodic-theoretic approach pursued in \cite{EMSS16}, this provides an explicit estimate on the rate of convergence. Note that the case $d=2$ also reduces to Kloosterman sums \cite{Mar10b,EMSS16}, but is significantly simpler. 

The following is our main result.

\begin{theorem}\label{thm:main_d=3}
Let $d=3$, $\epsilon>0$ and assume $f:\Gamma\bs\Gamma H \times \TT^{2}\to \R$ is infinitely differentiable with all derivatives bounded. Then there is a constant $C_{\epsilon,f}<\infty$ such that, for all $q\in\N$,
\begin{equation}\label{eq:main_d=3}
\bigg| \frac{1}{\#\scrR_q} \sum_{\vecr\in\scrR_q} f(\Gamma n_+(\vecr)D(q),\vecr)  - \int_{\Gamma\bs\Gamma H\times\TT^{2}} f(g,\vecx) \, d\mu_0(g) \, d\vecx \bigg| \leq C_{\epsilon,f} \, q^{-\frac{1}{2}+\theta+\epsilon}.
\end{equation}
\end{theorem}
Here $\theta$ is the constant towards the Ramanujan conjecture, which asserts $\theta=0$. 
The best current bound is $7/64$ due to Kim and Sarnak \cite[App.~2]{Kim03}. 

Theorem \ref{thm:main_d=3} complements Ustinov's effective results on the distribution of Frobenius numbers in three variables \cite{Ust10}. In view of \cite{Mar10,MS13}, these can be obtained by choosing a particular class of test functions in \eqref{eq:main_d=3}. These test functions are, however, not infinitely differentiable and therefore Theorem \ref{thm:main_d=3} does not directly imply Ustinov's bound; this would require further uniform estimates of the error term in \eqref{eq:main_d=3} to allow the approximation of singular functions by differentiable functions. We furthermore note that the test functions relevant to Frobenius numbers and circulant graphs are in fact invariant under the right action of the subgroup $\{(\begin{smallmatrix} 1_2 & \vecxi \\ 0 & 1\end{smallmatrix}) \;:\; \vecxi \in \R^{2}\}$; cf. \cite{Mar10,MS13}.

The plan of this paper is as follows. In Section \ref{sec:rational} we give an explicit representation of the rational points on a large two-dimensional horosphere in terms of natural coordinates of the subgroup $H$. This representation, combined with Fourier analysis on $H\simeq\ASL_2(\R)$ (Section \ref{sec:Fourier}), allows us in Section \ref{sec:main_d=3} to separate the proof of Theorem \ref{thm:main_d=3} into (a) an equidistribution problem on $\SL_2(\Z)\bsl\SL_2(\R)$ and (b) estimates of Kloosterman sums. Part (a) reduces to spectral gap estimates for Hecke operators in the case of uniform weights on the rational points (Section \ref{sec:prop4.1}) and Ramanujan sums in the case of non-uniform weights (Section \ref{sec:prop4.2}); estimates for (b) reduce to Weil's classic bounds for Kloosterman sums (Section \ref{sec:prop4.3}).

\section{Rational points on horospheres}\label{sec:rational}

Let $\ell^2$ denote the square-part of $q$, i.e., $\ell$ is the largest integer such that $\ell^2\mid q$. 

\begin{lemma}\label{lem:Rq}
We have
\begin{equation}
	\#\scrR_q 
	= \varphi(q) \sum_{a|q} \frac{a}{\gcd(a,q/a)}\; \varphi\big(\gcd(a,q/a)\big) 
	=\varphi(q) \sum_{d|\ell}  \mu(d)  \sigma_1(q/d^2), 
\end{equation}
and 
\begin{equation}\label{lowerR_q}
\#\scrR_q  > \frac{6}{\pi^2} \; q^2 .
\end{equation}
\end{lemma}

\begin{proof}
Note that 
\begin{align}
\begin{split}
	\#\scrR_q 
	& = \#\{ (p_1,p_2)\in\Z^2 : 0< p_1,p_2\leq q,\; \gcd(p_1,p_2,q)=1\} \\
	& =  \sum_{a|q} \#\big\{ (p_1,p_2)\in\Z^2 : 0< p_1,ap_2\leq q,\; \gcd(p_1,a)=1,\; \gcd(p_2,q/a)=1\big\} \\
	& = \sum_{a|q} \frac{q}{a}\,\varphi(a)\,\varphi\bigg(\frac{q}{a}\bigg) 
= \sum_{a|q} a\,\varphi(a)\,\varphi\bigg(\frac{q}{a}\bigg) \\
	& = \varphi(q) \sum_{a|q} \frac{a}{\gcd(a,q/a)}\; \varphi\big(\gcd(a,q/a)\big). 
\end{split}
\end{align}
Using the standard expansion of Euler's totient function in terms of the M\"obius function yields
\begin{align}
\#\scrR_q  
& = \varphi(q) \sum_{a|q} a\, \sum_{d|\gcd(a,q/a)}  \frac{\mu(d)}{d}.
\end{align}
We have $\gcd(a,q/a)|\ell$, and hence $d|\ell$. 
Thus, setting $a'=\frac{a}{d}$, 
we infer that
\begin{equation}
	\#\scrR_q  
	= \varphi(q) \sum_{d|\ell} \mu(d)  \sum_{a'| q/d^2} a'  
	= \varphi(q) \sum_{d|\ell}  \mu(d)  \sigma_1(q/d^2).
\end{equation}
The bound \eqref{lowerR_q} follows from
\begin{equation}
	\varphi(q) \sum_{d|\ell}  \mu(d)  \sigma_1(q/d^2) 
	\geq \varphi(q)   \sigma_1(q) 
	> \frac{6}{\pi^2} \; q^2.
\end{equation}
\end{proof}

\begin{lemma}\label{lem:inj}
For every $t\in\R$, the map
$$
	(0,1]^{d-1} \to \GamG, \qquad \vecx\mapsto \Gamma n_+(\vecx)\Phi^t
$$
is injective.
\end{lemma}

\begin{proof}
Since the subgroup $n_+(\Z^d) = \Gamma \cap n_+(\R^d)$
is a lattice in $n_+(\R^d)$, we have that $\Gamma\bs\Gamma n_+(\R^d)$ is an embedded submanifold isomorphic to $n_+(\Z^d)\bs n_+(\R^d)\simeq \TT^d:=\R^d/\Z^d$.
\end{proof}

Lemma \ref{lem:inj} implies in particular that there is a one-to-one correspondence between the elements of $\scrR_q$ and the set $\{ \Gamma n_+(\vecr) D(q):\vecr\in\scrR_q\}$.
In view of \eqref{ins}, for every $q\in\N$ and $\vecr\in\scrR_q$, there are $A\in \SL_{d-1}(\R)$ and $\vecs\in \R^{d-1}$ such that 
\begin{equation}\label{e:set-up}
	\Gamma n_+(\vecr) D(q)
	=
	\Gamma \bpm A& \vecs \\ \trans\vecnull & 1\ebpm
	.
\end{equation}

We make this relationship explicit in the case $d=3$:

\begin{lemma}\label{lem:sol_d=3}
Let $q\in\N$ and $\vecr=q^{-1} \vecp \in\scrR_q$ with $\vecp=\sm p_1 \\ p_2\esm\in\Z^2$. 
Set $a=\gcd(p_1, q)$, $q_0$ be the smallest positive divisor of $q$ such that $a|q_0$ and $\gcd(q_0, q/q_0)=1$ 
and $q_1=q/q_0$. Choose $\ell_0, \ell_1\in \Z$ such that $q_0\ell_0+q_1\ell_1=1$.
Then there exist uniquely determined $x_1, x_2\pmod{q_1}$, $y_1, y_2\pmod{q_0}$ with $\gcd(x_1, q_1)=1$, $\gcd(y_1, q_0)=a$ and $\gcd(y_2, q_0)=1$, such that 
\begin{equation}\label{e:p_xy}
	p_1\equiv x_1 q_0\ell_0+ y_1 q_1\ell_1\pmod{q} 
	\text{ and }
	p_2\equiv x_2 q_0\ell_0+ y_2q_1\ell_1\pmod{q}. 
\end{equation}
and Eq.~\eqref{e:set-up} holds with
$$
	(A,\vecs) =\left( q^{-\frac{1}{2}} \bpm a & b\\ 0 & \frac{q}{a}\ebpm , q^{-1} \bpm b_1 + h \frac{q}{a} \\  \frac{q}{a} b_2  \ebpm \right), 
$$
where the integers $b,b_1\in[0,\frac{q}{a})$, $b_2,h\in[0,a)$ are uniquely defined by
\begin{align}
	& b_1\equiv a\overline{x_1} q_0\ell_0 + \tilde{y}_1 q_1\ell_1\pmod{q/a}, \\
	& b_2\equiv \overline{y_2} \pmod{a}, \\
	& h\equiv \overline{y_2} \frac{b-\tilde{y}_1y_2}{q_0/a} \pmod{a} \\
	& b\equiv a\overline{x_1} x_2 q_0\ell_0+ \tilde{y}_1 y_2 q_1\ell_1 \pmod{q/a} \label{e:b_lem}
\end{align}
Here $x_1\overline{x_1}\equiv 1\pmod{q_1}$, $\frac{y_1}{a} \tilde{y}_1\equiv 1\pmod{q_0/a}$ and $y_2 \overline{y_2} \equiv 1\pmod{q_0}$. 
%
%
\end{lemma}


\begin{proof}
From \eqref{e:set-up}, for general $d$ and $\vecp=q\vecr\in\Z^{d-1}$,
\begin{align*}
	\bpm A & \vecs \\  \trans\vecnull & 1\ebpm 
	\bpm q^{-\frac{1}{d-1}} 1_{d-1} & \vecnull\\ \trans\vecnull & q\ebpm \bpm 1_{d-1} &  \vecnull \\ -\trans\vecr & 1\ebpm 
	& =
	\bpm A& \vecs \\  \trans\vecnull & 1\ebpm 
	\bpm q^{-\frac{1}{d-1}} 1_{d-1} &  \vecnull \\ -\trans\vecp & q\ebpm
	\\
	& =
	\bpm q^{-\frac{1}{d-1}} A -\vecs\trans\vecp  & q\vecs \\ -\trans\vecp & q\ebpm
	\in \SL_d(\Z). 
\end{align*}
This implies in particular
\begin{equation}\label{e:cond1}
	q\vecs \in \Z^{d-1}, \qquad q^{-1} \left( q^{\frac{d-2}{d-1}} A - q\vecs \trans\vecp\right) \in M_{d-1\times d-1}(\Z),
\end{equation}
and hence
\begin{equation}\label{e:cond2}
	q^{\frac{d-2}{d-1}} A \in M_{d-1\times d-1}(\Z) , \qquad q^{\frac{d-2}{d-1}} A \equiv  q\vecs\trans\vecp\mod q. 
\end{equation}
Also note that $\det(q^{\frac{d-2}{d-1}} A) = q^{d-2}\det(A) = q^{d-2}$. 

We now specialize to $d=3$. 
Let $q^{\frac{1}{2}}A=\sm a & b\\ c& d\esm \in M_{2\times 2}(\Z)$. 
If $c\neq 0$ and $a=0$, we have
$$
	\bpm 0 & -1 \\ 1 & 0 \ebpm q^{\frac{1}{2}} A 
	=
	\bpm -c & -d\\0 & b\ebpm
	.
$$
If $c\neq 0$ and $a\neq 0$, we have $\sm * & * \\ -\frac{c}{\gcd(a, c)} & \frac{a}{\gcd(a, c)}\esm \in \SL_2(\Z)$ and 
$$
	\bpm * & * \\ -\frac{c}{\gcd(a, c)} & \frac{a}{\gcd(a, c)}\ebpm 
	q^{\frac{1}{2}} A
	=
	\bpm * & * \\ 0 & * \ebpm \in M_{2\times 2}(\Z). 
$$
This proves that, by replacing $A$ by $\gamma A$ (and $\vecs$ by $\gamma\vecs$) for some $\gamma\in \SL_2(\Z)$, we can assume without loss of generality that $q^{\frac{1}{2}} A=\sm a & b \\ 0 & d\esm$. 
Now $\det(q^{\frac{1}{2}} A)=q \det(A) = q=ad$, so
$$
	q^{\frac{1}{2}} A = \bpm a & b\\ 0& \frac{q}{a}\ebpm
	\qquad \text{ with } a\mid q. 
$$
Furthermore, for $\gamma=\sm 1 & m \\ 0 & 1 \esm\in\SL_2(\Z)$ we have $q^{\frac{1}{2}}\gamma A= \sm a & b+m\frac{q}{a}\\ 0& \frac{q}{a}\esm$, which shows that (again by replacing $(A,\vecs)$ with $(\gamma A,\gamma\vecs)$) we may choose the representative $0\leq b < \frac{q}{a}$. 
Noting that $\sm 1_2 & \vecm \\ 0 & 1\esm\in\Gamma$ with $\vecm\in\Z^2$, we see that a representative of $\vecs$ can be chosen in $[0,1)^2$.

With $q\vecs=:\sm t_1 \\ t_2\esm$, Rel.~\eqref{e:cond2} becomes
$$
	\bpm a & b\\0 & \frac{q}{a}\ebpm 
	\equiv \bpm t_1 p_1 & t_1 p_2 \\ t_2 p_1 & t_2p_2\ebpm \mod{q} ,
$$
which we write as the system of equations
\begin{enumerate}[(i)]
	\item $t_1p_1\equiv a \pmod{q}$,
	\item $t_2p_1 \equiv 0 \pmod{q}$,
	\item $t_1p_2\equiv b\pmod{q}$, 
	\item $t_2p_2\equiv \frac{q}{a}\pmod{q}$.
\end{enumerate}
Note that by the above choice of representatives, we are guaranteed that (iii) has a solution with $0\leq b <\frac{q}{a}$.

For $a|q$, set $q_0$ to be the smallest divisor of $q$ such that $a|q_0$ and $\gcd(q_0, \tfrac{q}{q_0})=1$. 
Let $q_1=\frac{q}{q_0}$. 
Since $\gcd(q_0, q_1)=1$, there exist $\ell_0, \ell_1\in \Z$ such that $q_0\ell_0+q_1\ell_1=1$. 

By the Chinese Remainder Theorem, for $p_1, p_2\in (0, q]$, there exist $x_1, x_2\pmod{q_1}$ and $y_1, y_2\pmod{q_0}$ such that 
$$
	p_1\equiv x_1 q_0\ell_0 + y_1 q_1\ell_1 \pmod{q} 
	\text{ and }
	p_2\equiv x_2 q_0\ell_0 + y_2 q_1\ell_1 \pmod{q}.
$$
Then 
\begin{align}
	& x_1 t_1 \equiv a\pmod{q_1}, \qquad y_1 t_1 \equiv a \pmod{q_0}, \label{e:xy1_t1} \\
	& x_1 t_2 \equiv 0 \pmod{q_1}, \qquad y_1 t_2 \equiv 0\pmod{q_0}, \label{e:xy1_t2} \\
	& x_2 t_2\equiv 0\pmod{q_1}, \qquad y_2t_2 \equiv \frac{q_0}{a} q_1 \pmod{q_0}, \label{e:xy2_t2} \\
	& x_2 t_1 \equiv b\pmod{q_1}, \qquad y_2 t_1 \equiv b\pmod{q_0}. \label{e:xy2_t1}
\end{align}
Since $\gcd(a, q_1)=1$, by the first part of \eqref{e:xy1_t1}, we get 
\begin{equation}\label{e:x1t1_q1}
	\gcd(x_1, q_1)=\gcd(t_1, q_1)=1. 
\end{equation}
From the second part of \eqref{e:xy1_t2}, we get $\frac{q_0}{\gcd(q_0, y_1)} \mid t_2$, then 
\begin{equation}\label{e:q0y1_q0t2}
	\frac{q_0}{\gcd(q_0, y_1)} \mid \gcd(t_2, q_0). 
\end{equation}
From the second part of \eqref{e:xy2_t2}, we again get $\gcd(t_2, q_0) \mid \frac{q_0}{a}$. 
Combining with \eqref{e:q0y1_q0t2}, $\frac{q_0}{\gcd(q_0, y_1)}\mid \frac{q_0}{a}$, so 
$a\mid \gcd(q_0, y_1)$. 
By the second part of \eqref{e:xy1_t1}, $\gcd(q_0, y_1)\mid a$. 
Therefore, $\gcd(q_0, y_1)=a$. 
Moreover, $y_2 \frac{t_2}{q_0/a}\equiv q_1\pmod{a}$, we get $\gcd(y_2, a)=1$. 
Note that, since every prime dividing $q_0$ also divides $a$, $\gcd(y_2, q_0)=1$. 

We get 
\begin{align}
	& t_1 \equiv a \overline{x_1} \pmod{q_1}, \qquad t_1 \equiv \tilde{y}_1 \pmod{\frac{q_0}{a}}, \label{e:t1_q}\\
	& t_2 \equiv 0\pmod{q_1}, \qquad t_2 \equiv \overline{y_2} q_1\frac{q_0}{a} \pmod{q_0} \label{e:t2_q}
\end{align}
where $\overline{x_1}$ is the multiplicative inverse of $x_1$ modulo $q_1$, $\tilde{y}_1$ is the multiplicative inverse of $\frac{y_1}{a}$ modulo $\frac{q_0}{a}$ 
and $\overline{y_2}$ is the multiplicative inverse of $y_2$ modulo $q_0$. 
Take $h\in [0, a)$ and set $t_1\equiv \tilde{y}_1 + h\frac{q_0}{a}\pmod{q_0}$. 
We will determine $h$ later. 
By the Chinese remainder theorem, we get
\begin{align}
	& t_1\equiv a\overline{x_1} \ell_0q_0 + \left(\tilde{y}_1+h\frac{q_0}{a}\right) \ell_1 q_1\pmod{q}, \label{e:t1}\\
	& t_2 \equiv \overline{y_2} \frac{q}{a} \ell_1 q_1 \pmod{q}. \label{e:t2}
\end{align}

Recalling \eqref{e:xy2_t1}, combining with \eqref{e:t1_q}, and then again by the Chinese remainder theorem, we get
\begin{equation}
	b\equiv a\overline{x_1} x_2  \ell_0 q_0 + y_2\tilde{y}_1 \ell_1q_1\pmod{\frac{q}{a}}. 
\end{equation}
The second part of \eqref{e:xy2_t1} yields
\begin{equation}
	b\equiv \tilde{y}_1y_2 + h y_2 \frac{q_0}{a}  \pmod{q_0}.
\end{equation}
So $h\in [0, a)$ can be determined by 
\begin{equation}
	h\equiv \frac{b-\tilde{y}_1 y_2}{q_0/a} \overline{y_2} \pmod{a}. 
\end{equation}

By \eqref{e:t2}, we also get
\begin{equation}
	\frac{t_2}{q/a} \equiv \overline{y_2}\pmod{a}, 
\end{equation}
since $\ell_1q_1\equiv 1\pmod{a}$. 

\end{proof}

It will be convenient to change the above parametrization of the solutions of \eqref{e:set-up} slightly.
For $a\mid q$, set
\begin{equation}\label{e:Rqa}
	\scrQ_{q, a}
	:=
	\left\{(c_1, c_2, b)\;\bigg|\; 
	{c_1\in (0, q_1], \; c_2\in (0, q_0], b\in (0, q/a], \atop 
	\gcd(c_1, q_1)=\gcd(c_2, q_0)=\gcd(b, \gcd(q/a, a))=1}\right\},
\end{equation}
where $q_0,q_1$ are as defined in Lemma \ref{lem:sol_d=3}. Note that
\begin{equation}
\#\scrQ_{q, a} = \varphi(q)\, \frac{q/a}{\gcd(q/a, a)}\, \varphi(\gcd(q/a, a)),
\end{equation}
and compare with Lemma \ref{lem:Rq}.


For $(c_1, c_2, b)\in \scrQ_{q, a}$, 
let $x_1\equiv c_1a \pmod{q_1}$ and $y_2 \equiv c_2\pmod{q_0}$.
By \eqref{e:xy2_t1} and \eqref{e:t1_q}, 
$$
	x_2\equiv b c_1\pmod{q_1} 
$$
and
$$
	\frac{y_1}{a} \equiv \overline{b} c_2 \pmod{q_0/a},
$$
where $\overline{b}$ is defined as the inverse of $b\pmod{q_0/a}$.
Furthermore, set
\begin{align}
	& p_1\equiv a(c_1 q_0\ell_0 + \overline{b} c_2 q_1\ell_1) \pmod{q}, \label{e:p1_c1c2b}\\
	& p_2 \equiv bc_1 q_0\ell_0 + c_2 q_1\ell_1 \pmod{q} \label{e:p2_c1c2b}
\end{align}
(cf. \eqref{e:p_xy}), 
\begin{align}
	& b_1\equiv \overline{c_1} q_0\ell_0 + b\overline{c_2} q_1\ell_1 \pmod{q/a},\\
	& b_2\equiv \overline{c_2}\pmod{a}, \label{e:b2_c1c2b}\\
	& h\equiv b\overline{c_2}  \frac{1-\overline{c_2} c_2}{q_0/a}\pmod{a}
\end{align}
with
\begin{equation}
	c_1\overline{c_1}\equiv 1\pmod{q_1}, \qquad c_2\overline{c_2} \equiv 1\pmod{q_0}.
\end{equation}
Note that 
\begin{equation}\label{e:b1_c1c2b}
	b_1+h\frac{q}{a} 
	=
	\overline{c_1} q_0\ell_0 + b\overline{c_2} q_1\ell_1
	+
	b\overline{c_2} (1-\overline{c_2}c_2) q_1
	\equiv 
	\overline{c_1} q_0\ell_0 + b\overline{c_2} q_1\ell_1
	\pmod{q}. 
\end{equation}
This yields the following reformulation of Lemma \ref{lem:sol_d=3}.

\begin{lemma}\label{lem:alt}
For every $q\in\N$, the map
\begin{align*}
\bigcup_{a|q} \scrQ_{q,a} & \to \{ \Gamma n_+(\vecr) D(q) \mid \vecr\in\scrR_q \} \\
(c_1,c_2,b) & \mapsto \Gamma \bpm A& \vecs \\ \trans\vecnull & 1\ebpm
\end{align*}
is bijective, where
$$
	A=q^{-\frac{1}{2}} \bpm a & b\\ 0 & \frac{q}{a}\ebpm ,
	\qquad
	\vecs = q^{-1} \bpm \overline{c_1} q_0\ell_0 + b\overline{c_2} q_1\ell_1 \\ \frac{q}{a} \overline{c_2}\ebpm ,
$$
and $q_0,q_1,\ell_0,\ell_1$ are defined as in Lemma \ref{lem:sol_d=3}.
\end{lemma}

\section{Fourier decomposition on $\ASL_2(\Z) \bsl \ASL_2(\R)$}\label{sec:Fourier}

In this section, we follow the argument given in \cite[\S4]{Str15}. Note that \cite{Str15} uses a different representation of $\ASL_2(\R)$, so some care has to be taken in translating the relevant results to the present setting.

Take $F\in \Cb_b^k (\ASL_2(\Z) \bsl \ASL_2(\R))$, the space of $k$ times continuously differentiable functions with all derivatives bounded. 
For $F\in \Cb_b^k(\ASL_2(\Z) \bsl \ASL_2(\R))$ we set
$$
	\|F\|_{\Cb_b^k} := \sum_{\ord(D) \leq k} \|DF\|_{L^\infty}, 
$$
the sum being over every left invariant differential operator $D$ on $\ASL_2(\R)$ of degree $\leq k$. 
For a fixed $A\in \SL_2(\R)$, for $\vecxi\in \R^2$ and $\vecm\in \Z^2$,  
we have $(1_2, \vecm)(A, \vecxi) = (A, \vecm)$. 
Since $(1_2, \vecm)\in\ASL_2(\Z) $, 
$$
	F(A, \vecxi+\vecm) = F(A, \vecm).
$$
Hence $F$ is periodic as a function of $\vecxi\in \R^2$ and we have the following Fourier expansion: 
\begin{equation}\label{e:F_Fourier}
	F(A, \vecxi)
	=
	\widehat{F}_\vecnull(A) 
	+
	\sum_{\vecm\in \Z^2\setminus\{\vecnull\}}
	\widehat{F}_\vecm(A)
	e(\trans\vecm\vecxi ). 
\end{equation}
Here $e(x) := e^{2\pi ix}$ and, 
\begin{equation}\label{e:F_FoureirT}
	\widehat{F}_\vecm(A)
	:=
	\int_{\R^2/\Z^2} F(A, \vecxi) e(-\trans \vecm\vecxi) \;d\vecxi.
\end{equation}
For $\vecm=\sm m_1 \\ m_2\esm\in\Z^2$,  we also write $\widehat{F}_{m_1,m_2}(A):=\widehat{F}_\vecm(A)$, and furthermore,
for $m\in\Z$, let $\widehat{F}_m(A):= \widehat{F}_{\sm 0\\ m\esm}(A)$. 
We have the following lemmas; see Lemmas 4.1 and 4.2 in \cite{Str15}. 

\begin{lemma}\label{lem:F_Fourier}
Let $F\in \Cb_b^2 (\ASL_2(\Z) \bsl \ASL_2(\R))$.
For any $\gamma\in \SL_2(\Z)$, we have
\beq\label{eq:o1}
	\widehat{F}_\vecm(\gamma A) = \widehat{F}_{\trans\gamma\vecm}(A),
\eeq
and, in particular, 
\beq\label{eq:o2}
	\widehat{F}_0(\gamma A) = \widehat{F}_0(A). 
\eeq
Moreover, for any $\vecxi=\sm \xi_1 \\ \xi_2\esm\in\R^2$,
\begin{equation}\label{e:F_Fourier_2}
	F(A, \vecxi)
	=
	\widehat{F}_0(A) 
	+
	\sum_{m=1}^\infty \sum_{(c, d)\in \widehat{\Z}^2} 
	\widehat{F}_m\left(\bpm * & * \\ c& d\ebpm A\right)
	e(cm\xi_1+dm\xi_2),
\end{equation}
where $\sm * & * \\ c & d\esm$ denotes an arbitrary choice of matrix $\sm a & b \\ c & d\esm\in \SL_2(\Z)$ for given $(c, d)\in \widehat{\Z}^2$. The sum in \eqref{e:F_Fourier_2} is absolutely convergent, uniformly in compacta.
\end{lemma}
%
%
%

\begin{lemma}\label{lem:Fourier_est}
If $F\in \Cb_b^k (\ASL_2(\Z) \bsl \ASL_2(\R))$, then
\begin{equation}\label{e:Fourier_est}
	\left|\widehat{F}_m(A)\right|
	\ll_k
	\frac{\|F\|_{\Cb_b^k}}{\max\{|mc|^{k}, |md|^{k}\}} 
\end{equation}
uniformly for all $m\in\N$ and $A=\sm a& b\\ c & d\esm \in \SL_2(\R)$.
\end{lemma}
\section{Proof of Theorem \ref{thm:main_d=3}} \label{sec:main_d=3}

Assume $f$ as in Theorem \ref{thm:main_d=3}, i.e., $f\in\C_b^\infty(\Gamma\bsl\Gamma H\times\TT^2)$.
We have the Fourier expansion
$$
	f(g, \vecx)
	=
	\sum_{n_1, n_2\in \Z} \widehat{f}_{n_1, n_2}(g) e(n_1x_1+n_2x_2), 
$$
for $\vecx = (x_1, x_2)$. 
Here 
$$
	\widehat{f}_{n_1, n_2}(g)
	=
	\int_{\R^2/\Z^2}
	f(g; x_1, x_2) e(-n_1x_1-n_2x_2) \, dx_1 \, dx_2. 
$$
By applying integration by parts repeatedly, we have
\begin{equation}\label{e:Fourier_f_1}
	\sup_{g\in\Gamma\bsl\Gamma H}\left|\widehat{f}_{n_1, n_2}(g)\right|
	\ll_{f, k} 
	(1+|n_1|+|n_2|)^{-k}
\end{equation}
for any positive integer $k\geq 1$. 

For each given $\vecr=q^{-1}\sm p_1\\ p_2\esm\in \scrR_q$,  choose $A=A(\vecr) \in \SL_2(\R)$ and $$\vecs=\vecs(\vecr)=q^{-1}\bpm b_1+h\frac{q}{a}\\ \frac{q}{a} b_2\ebpm \in q^{-1}\Z^2$$ 
as in  Lemma \ref{lem:sol_d=3}.
Then
\begin{multline*}
	\sum_{\vecr\in \scrR_q}
	f\left(\Gamma n_+(\vecr) D(q), \vecr\right) 
	=
	\sum_{\vecr\in \scrR_q}
	\widehat{f}_{0,0}\left(\bpm A(\vecr) & \vecs(\vecr) \\ \trans\vecnull & 1\ebpm\right)
	\\
	+
	\sum_{\vecr=q^{-1}\sm p_1\\p_2\esm \in \scrR_q}
	\sum_{\substack{n_1, n_2\in \Z \\ (n_1, n_2)\neq 0}}
	\widehat{f}_{n_1, n_2} \left(\bpm A(\vecr) & \vecs(\vecr) \\ \trans\vecnull & 1\ebpm\right)
	e\left(n_1 \tfrac{p_1}{q} + n_2 \tfrac{p_2}{q}\right).
\end{multline*}
By \eqref{e:Fourier_f_1}, for $\epsilon>0$, 
\begin{multline}\label{e:est1}
	\frac{1}{\# \scrR_{q}} \sum_{\vecr\in \scrR_q}
	f\left(\Gamma n_+(\vecr) D(q), \vecr\right) 
	=
	\frac{1}{\# \scrR_{q}} 
	\sum_{\vecr=q^{-1} \sm p_1\\p_2\esm \in \scrR_q}
	\widehat{f}_{0,0}\left(\bpm A(\vecr) & \vecs(\vecr) \\ \trans\vecnull & 1\ebpm\right)
	\\
	+
	\frac{1}{\# \scrR_{q}} 
	\sum_{\substack{n_1, n_2\in \Z \\ (n_1, n_2)\neq 0\\ |n_1|, |n_2| \leq q^{\epsilon}}}
	\sum_{\vecr=q^{-1}\sm p_1\\ p_2\esm \in \scrR_q}
	\widehat{f}_{n_1, n_2} \left(\bpm A(\vecr) & \vecs(\vecr) \\ \trans\vecnull & 1\ebpm\right)
	e\left(n_1 \tfrac{p_1}{q} + n_2 \tfrac{p_2}{q}\right)
	+
	O_{f,k}(q^{-\epsilon k}) .
\end{multline}
Note that for any $(n_1, n_2)\in \Z^2$, 
\begin{equation}
	F_{n_1, n_2}(A, \vecxi)
	:=
	\widehat{f}_{n_1, n_2} \left(\bpm A& \vecxi\\  \trans\vecnull& 1\ebpm \right), 
\end{equation}
defines a function on $\ASL_2(\Z) \bsl \ASL_2(\R)$. 
Since $f\in\C_b^\infty(\Gamma\bsl\Gamma H\times\TT^2)$, we have $\widehat{f}_{n_1, n_2}\in\C_b^\infty(\Gamma\bsl\Gamma H)$ and thus $F_{n_1, n_2}\in\C_b^\infty(\ASL_2(\Z) \bsl \ASL_2(\R))$.  To simplify notation, we will drop the indices $n_1,n_2$ in the following and simply write $F:=F_{n_1, n_2}$.

For any given positive integer $q\geq 1$, our goal is to estimate 
\begin{equation}\label{e:S_Fq}
	S_{n_1, n_2} (F; q)
	:=
	\sum_{\vecr=q^{-1}\sm p_1\\ p_2\esm \in \scrR_q}
	F(A(\vecr), \vecs(\vecr))
	e(n_1\tfrac{p_1}{q}+n_2\tfrac{p_2}{q}), 
\end{equation}
for $n_1, n_2\in \Z$ such that $|n_1|, |n_2| \leq q^{\epsilon}$. 

For any $a|q$, set $q_0, q_1, \ell_0$ and $\ell_1$ as defined in Lemma \ref{lem:sol_d=3}. 
In view of Lemma \ref{lem:alt},
\begin{multline}\label{e:S_Fq_general}
	S_{n_1, n_2}(F; q)
	=
	\sum_{a|q}
	\sum_{(c_1, c_2, b)\in \scrQ_{q, a}}
	F\left(q^{-\frac{1}{2}} \bpm a & b\\ & \frac{q}{a}\ebpm, 
	\bpm \frac{\overline{c_1} q_0\ell_0 + \overline{c_2} b q_1\ell_1}{q} \\[5pt] 
	\frac{\overline{c_2}}{a} \ebpm \right)
	\\
	\times
	e\left(n_1\tfrac{c_1q_0\ell_0 + c_2 \overline{b} q_1\ell_1}{q/a} 
	+ n_2 \tfrac{bc_1q_0\ell_0 + c_2q_1\ell_1}{q}\right).
\end{multline}
Using the Fourier expansion of $F$ in \eqref{e:F_Fourier_2}, we obtain
\begin{multline*}
	S_{n_1, n_2}(F; q)
	=
	\sum_{a\mid q}
	\sum_{(c_1, c_2, b)\in \scrQ_{q, a}}
	\widehat{F}_0 \left(\bpm \frac{a}{\sqrt{q}} & \frac{b}{\sqrt{q}} \\ 0 & \frac{\sqrt{q}}{a}\ebpm\right)
	e\left(n_1\tfrac{c_1q_0\ell_0 + c_2 \overline{b} q_1\ell_1}{q/a} 
	+ n_2 \tfrac{bc_1q_0\ell_0 + c_2q_1\ell_1}{q}\right)
	\\
	+
	\sum_{a\mid q}
	\sum_{(c_1, c_2, b)\in \scrQ_{q, a}}
	\sum_{m=1}^\infty 
	\sum_{\substack{c, d\in \Z \\ \gcd(c, d)=1}}
	\widehat{F}_{m}
	\left(\bpm * & * \\ c& d\ebpm \bpm \frac{a}{\sqrt{q}} & \frac{b}{\sqrt{q}} \\ 0 & \frac{\sqrt{q}}{a}\ebpm\right)
	\\
	\times
	e\left(cm \tfrac{\overline{c_1} q_0 \ell_0+ \overline{c_2} bq_1\ell_1}{q} 
	+ dm \tfrac{\overline{c_2}}{a}\right)
	e\left(n_1\tfrac{c_1q_0\ell_0 + c_2 \overline{b} q_1\ell_1}{q/a} 
	+ n_2 \tfrac{bc_1q_0\ell_0 + c_2q_1\ell_1}{q}\right).
\end{multline*}
Note that
\begin{multline}\label{e:Klo_0}
	\sum_{\substack{c_1\pmod{q_1} \\ 
	\gcd(c_1,q_1)=1 }}
	\sum_{\substack{c_2\pmod{q_0} \\ 
	\gcd(c_2,q_0)=1}}
	e\left(n_1\tfrac{c_1q_0\ell_0 + c_2 \overline{b} q_1\ell_1}{q/a} 
	+ n_2 \tfrac{bc_1q_0\ell_0 + c_2q_1\ell_1}{q}\right)
	\\
	=
	c_{q_1}(n_1a+n_2b)
	c_{q_0}(n_1\overline{b} a+ n_2), 
\end{multline}
where
$$
	c_r(n) = \sum_{\substack{\alpha\pmod{r} \\ \gcd(\alpha, r)=1}} e\left(n\tfrac{\alpha}{r}\right)
$$
is the Ramanujan sum, 
and furthermore
\begin{multline}\label{e:Klo_m}
	\sum_{\substack{c_1\pmod{q_1} \\ 
	\gcd(c_1,q_1)=1 }}
	\sum_{\substack{c_2\pmod{q_0} \\ 
	\gcd(c_2,q_0)=1}}
	e\left(cm \tfrac{\overline{c_1} q_0 \ell_0+ \overline{c_2} bq_1\ell_1}{q} 
	+ dm \tfrac{\overline{c_2}}{a}\right)
	e\left(n_1\tfrac{c_1q_0\ell_0 + c_2 \overline{b} q_1\ell_1}{q/a} 
	+ n_2 \tfrac{bc_1q_0\ell_0 + c_2q_1\ell_1}{q}\right)
	\\
	=
	S\left( (n_1a+n_2b) \ell_0, cm\ell_0; q_1\right) 
	S\left( (n_1\overline{b} a+ n_2) \ell_1, (cb\ell_1 + d\tfrac{q_0}{a})m ; q_0 \right).
\end{multline}
Here
$$
	S(n, m; r) = \sum_{\substack{\alpha\pmod{r} \\ \gcd(\alpha,r)=1}} 
	e\left(n\tfrac{\alpha}{r} + m \tfrac{\overline{\alpha}}{r}\right)
$$
is the Kloosterman sum.
We have of course  $c_r(n) = S(n, 0; r)$.

With this
\begin{multline}\label{e:S_Fourier_general}
	S_{n_1, n_2}(F; q)
	=
	\sum_{a\mid q}
	\sum_{\substack{b\pmod{q/a} \\ \gcd(b, \gcd(q/a, a))=1}}
	\widehat{F}_0 \left(q^{-\frac{1}{2}} \bpm a & b\\ 0 & \frac{q}{a}\ebpm\right)
	c_{q_1}(n_1a+n_2b)
	c_{q_0}(n_1\overline{b} a+ n_2)
	\\
	+
	\sum_{a\mid q}
	\sum_{\substack{b\pmod{q/a} \\ \gcd(b, \gcd(q/a, a))=1}}
	\sum_{m=1}^\infty 
	\sum_{\substack{c, d\in \Z \\ \gcd(c, d)=1}}
	\widehat{F}_{m}
	\left(\bpm * & * \\ c& d\ebpm q^{-\frac{1}{2}} \bpm a & b\\ 0 & \frac{q}{a}\ebpm\right)
	\\
	\times
	S\left( (n_1a+n_2b) \ell_0, cm\ell_0; q_1\right) 
	S\left( (n_1\overline{b} a+ n_2) \ell_1, (cb\ell_1 + d\tfrac{q_0}{a})m ; q_0 \right).
\end{multline}
For $m\geq 1$, by \eqref{e:Fourier_est}, 
$$
	\left|\widehat{F}_m\left(\bpm *&* \\ c& d\ebpm \bpm \frac{a}{\sqrt{q}} & \frac{b}{\sqrt{q}} \\ & \frac{\sqrt{q}}{a}\ebpm\right)\right|
	\ll_k
	\frac{\|F\|_{\Cb_b^k}}
	{\left(|mc\frac{a}{\sqrt{q}}|^{2}+|m\left(c\frac{b}{\sqrt{q}} + d\frac{\sqrt{q}}{a}\right)|^{2}\right)^{\frac{k}{2}}} 
	.
$$
Let 
\begin{multline}
	E(F; q)
	:=
	\sum_{a|q}
	\sum_{\substack{b\pmod{q/a} \\ \gcd(b, \gcd(q/a, a))=1}}
	\sum_{m=1}^\infty 
	\sum_{\substack{c, d\in \Z \\ \gcd(c, d)=1 \\ 
	\max\left\{m|c|\frac{a}{\sqrt{q}}, m\left|\frac{bc}{q/a}+d\right|\frac{\sqrt{q}}{a}\right\} > q^\epsilon}}
	\widehat{F}_{m}
	\left(\bpm * & * \\ c& d\ebpm q^{-\frac{1}{2}} \bpm a & b\\ 0 & \frac{q}{a}\ebpm\right)
	\\
	\times
	S\left( (n_1a+n_2b) \ell_0, cm\ell_0; q_1\right) 
	S\left( (n_1\overline{b} a+ n_2) \ell_1, (cb\ell_1 + d\tfrac{q_0}{a})m ; q_0 \right) .
\end{multline}
Applying the trivial bound of the Kloosterman sums, we see that
\begin{multline}\label{e:Efq0}
	\left|E(F;q)\right|
	\ll_k 
	\|F\|_{C_b^k} \varphi(q)
	\sum_{a|q}
	\sum_{\substack{0\leq b< q/a \\ \gcd(b, \gcd(a, q/a))=1}}
	\\
	\times
	\sum_{\substack{m\in \Z_{\geq 1}, \; c, d\in \Z \\ \gcd(c, d)=1 
	\\ 
	\max\left\{m|c|\frac{a}{\sqrt{q}}, m\left|\frac{bc}{q/a}+d\right|\frac{\sqrt{q}}{a}\right\} > q^\epsilon}}
	\left(m\left(\frac{a^2}{q}|c|^2+\frac{q}{a^2}\left|\frac{bc}{q/a} + d\right|^{2}\right)^{\frac{1}{2}}\right)^{-k}
	.
\end{multline}
Let 
\begin{equation}
	\mathcal{E}_{a, q, b}(x, y):= (a^2+b^2) x^2+ 2b\frac{q}{a} xy + \frac{q^2}{a^2} y^2
	, 
\end{equation}
then 
$$
	\frac{a^2}{q}|c|^2+\frac{q}{a^2}\left|\frac{bc}{q/a} + d\right|^{2}
	=
	\frac{1}{q} \mathcal{E}_{a, q, b}(c, d)
	.
$$
Let 
\begin{equation}
	N_{a, q, b}(t) := 
	\# \left\{ (x, y) \in \Z^2 \;:\; \mathcal{E}_{a, q, b}(x, y) \leq t\right\}.
\end{equation}
It is well known that 
\begin{equation}
	N_{a, q, b}(t) = \frac{1}{q}O(t), 
\end{equation}
where the implied constant is independent of $a, q$ and $b$. 
It follows that 
\begin{multline}
	\sum_{\substack{m\in \Z_{\geq 1}, \; c, d\in \Z \\ \gcd(c, d)=1 
	\\ 
	\max\left\{m|c|\frac{a}{\sqrt{q}}, m\left|\frac{bc}{q/a}+d\right|\frac{\sqrt{q}}{a}\right\} > q^\epsilon}}
	\left(m\left(\frac{a^2}{q}|c|^2+\frac{q}{a^2}\left|\frac{bc}{q/a} + d\right|^{2}\right)^{\frac{1}{2}}\right)^{-k}
	\leq 
	\sum_{
	\substack{m\in \Z_{\geq 1}, c, d\in \Z, \\
	t=\mathcal{E}_{a, q, b}(c, d)\in \Z, \\
	\frac{m^2 t}{q} > q^{2\epsilon}}}
	\left(m \frac{\sqrt{t}}{\sqrt{q}}\right)^{-k}
	\\
	\leq 
	\sum_{m=1}^\infty 
	\left(\frac{m}{\sqrt{q}}\right)^{-k}
	\int_{\frac{q^{1+2\epsilon}}{m^2}}^\infty
	t^{-\frac{k}{2}} dN_{a, q, b}(t)
	\ll
	q^{2\epsilon-\epsilon k}
	\sum_{m=1}^\infty 
	\frac{1}{m^2}
	=
	q^{2\epsilon-\epsilon k }\zeta(2). 
\end{multline}
Applying this to \eqref{e:Efq0}, we get
\begin{equation}\label{e:Efq}
	\left|E(F;q)\right|
	\ll_k 
	\|F\|_{C_b^k} \varphi(q)
	q^{2\epsilon-\epsilon k} 
	\sum_{a|q}
	\sum_{\substack{0\leq b< q/a \\ \gcd(b, \gcd(a, q/a))=1}}
	1.
\end{equation}
In view of Lemma \ref{lem:Rq},
$$
	\varphi(q)
	\sum_{a|q}
	\sum_{\substack{0\leq b< q/a \\ \gcd(b, \gcd(a, q/a))=1}} 1
	=
	\varphi(q) \sum_{a|q} \frac{q/a}{\gcd(a,q/a)}\; \varphi\big(\gcd(a,q/a)\big) 
	=
	\# \scrR_q ,
$$
which yields
\begin{equation}\label{e:S_small}
	\frac{1}{\# \scrR_q} \left|E(F; q)\right|
	\ll_{k, F} q^{2\epsilon-k\epsilon}. 
\end{equation}

Combining \eqref{e:S_Fourier_general} and \eqref{e:S_small}, 
\begin{multline}\label{e:S_Fq_general_est2}
	\frac{1}{\#\scrR_{q}} S_{n_1, n_2}(F; q)
	\\
	=
	\frac{1}{\#\scrR_{q}} 
	\sum_{a|q}
	\sum_{\substack{b\pmod{q/a} \\ \gcd(b, \gcd(a, q/a))=1}}
	\widehat{F}_0 \left(\bpm \frac{a}{\sqrt{q}} & \frac{b}{\sqrt{q}} 
	\\ 0 & \frac{\sqrt{q}}{a}\ebpm\right)
	c_{q_1}(n_1a+n_2b)
	c_{q_0}(n_1\overline{b} a+ n_2)
	\\
	+
	\frac{1}{\#\scrR_{q}} 
	\sum_{a|q}
	\sum_{\substack{b\pmod{q/a} \\ \gcd(b, \gcd(a, q/a))=1}}
	\sum_{m=1}^\infty 
	\sum_{\substack{c, d\in \Z \\ \gcd(c, d)=1 \\ 
	\max\left\{ m|c| \frac{a}{\sqrt{q}}, \; m\frac{\sqrt{q}}{a} \left|\frac{bc}{q/a}+d\right| \right\} \leq q^\epsilon}}
	\widehat{F}_{m}
	\left(\bpm * & * \\ c& d\ebpm \bpm \frac{a}{\sqrt{q}} & \frac{b}{\sqrt{q}}
	\\ 0 & \frac{\sqrt{q}}{a}\ebpm\right)
	\\
	\times
	S\left( (n_1a+n_2b) \ell_0, cm\ell_0; q_1\right) 
	S\left( (n_1\overline{b} a+ n_2) \ell_1, (cb\ell_1 + d\tfrac{q_0}{a})m ; q_0 \right)
	\\
	+
	O_{k, F} \left(q^{2\epsilon-k \epsilon}\right).
\end{multline}
Theorem \ref{thm:main_d=3} now follows from the following three propositions, which we will prove in the next sections. 
\begin{proposition}\label{prop:m=0_main}
For $n_1=n_2=0$, 
\begin{multline*}
	\frac{1}{\# \scrR_q}
	\sum_{a|q} 
	\sum_{\substack{0\leq b < q/a \\ \gcd(b, \gcd(a, q/a))=1}}
	\widehat{F}_0\left(\bpm \frac{a}{\sqrt{q}} & \frac{b}{\sqrt{q}} \\ & \frac{\sqrt{q}}{a}\ebpm\right)
	c_{q_0}(0) c_{q_1}(0)
	\\
	=
	\int_{\SL_2(\Z) \bsl \SL_2(\R)} 
	\widehat{F}_0(g) \; d\mu(g)
	+
	O_{F, \epsilon'} (q^{-\frac{1}{2}+\epsilon'+\theta}), 
\end{multline*}
for any $\epsilon'>0$. Here $\theta$ is the constant towards the Ramanujan conjecture. 
\end{proposition}

\begin{proposition}\label{prop:m=0_error}
For $|n_1|, |n_2|<q^\epsilon$ with $(n_1, n_2)\neq (0, 0)$, 
$$
	\frac{1}{\# \scrR_q} \bigg|
	\sum_{a|q}
	\sum_{\substack{b\pmod{q/a} \\ \gcd(b, \gcd(a, q/a))=1}}
	\widehat{F}_0 \left(\bpm \frac{a}{\sqrt{q}} & \frac{b}{\sqrt{q}} 
	\\ 0 & \frac{\sqrt{q}}{a}\ebpm\right)
	c_{q_1}(n_1a+n_2b)
	c_{q_0}(n_1\overline{b} a+ n_2)
	\bigg|
	\ll_F
	\sigma_0(q)^2 q^{-1+2\epsilon} . 
$$
\end{proposition}

\begin{proposition}\label{prop:mneq0}
For $|n_1|, |n_2| < q^\epsilon$, 
\begin{multline*}
	\frac{1}{\# \scrR_q}\bigg|
	\sum_{a|q} 
	\sum_{\substack{0\leq b < q/a \\ \gcd(b, \gcd(a, q/a))=1}}
	\sum_{m=1}^\infty
	\sum_{\substack{(c, d)\in \widehat{\Z}^2 \\ \max\left\{ m|c| \frac{a}{\sqrt{q}}, m\frac{\sqrt{q}}{a} \left|\frac{bc}{q/a} + d\right| \right\} < q^{\epsilon}}}
	\widehat{F}_m \left(\bpm * & * \\ c& d\ebpm \bpm \frac{a}{\sqrt{q}} & \frac{b}{\sqrt{q}} \\ & \frac{\sqrt{q}}{a}\ebpm\right)
	\\
	\times
	S\left( (n_1a+n_2b) \ell_0, cm\ell_0; q_1\right) 
	S\left( (n_1\overline{b} a+ n_2) \ell_1, (cb\ell_1 + d\tfrac{q_0}{a})m ; q_0 \right)
	\bigg|
	\ll_F
	\sigma_0(q)^2 q^{-\frac{1}{2}+3\epsilon}.
\end{multline*}
\end{proposition}

Applying Proposition \ref{prop:m=0_main}, Proposition \ref{prop:m=0_error} and Proposition \ref{prop:mneq0} to \eqref{e:S_Fq_general_est2}, we get 
\begin{equation}\label{e:Sn1n2_fin}
	\frac{1}{\# \scrR_q} S_{n_1, n_2}(F; q)
	=
	\delta_{n_1=n_2=0} \int_{\SL_2(\Z)\bsl \SL_2(\R)} \widehat{F}_0(g) \; d\mu(g) 
	+ 
	O_{F, \epsilon'}(q^{-\frac{1}{2}+\epsilon'+\theta}), 
\end{equation}
for any $\epsilon'>0$. 


Recall that $F = \widehat{f}_{n_1, n_2}$. 
Applying \eqref{e:Sn1n2_fin} to \eqref{e:est1}, 
and taking the summation over $n_1, n_2\in \Z$ with $|n_1|, |n_2|<q^\epsilon$, 
we obtain
$$
	\frac{1}{\# \scrR_q} 
	\sum_{\vecr\in \scrR_q}
	f\left(\Gamma n_+(\vecr) D(q), \vecr\right)
	=
	\int_{\Gamma\bsl \Gamma H \times \TT^{2}}
	f(g, \vecx) \; d\mu_0(g) \; d\vecx
	+
	O_{f, \epsilon}(q^{-\frac{1}{2}+\theta+\epsilon})
	, 
$$
for any $\epsilon>0$. This yields Theorem \ref{thm:main_d=3}.

\section{Proof of Proposition \ref{prop:m=0_main}} \label{sec:prop4.1}
Since $\gcd(q_0, q_1)=1$ and $c_{q_0}(0) c_{q_1}(0) = \varphi(q_0)\varphi(q_1) = \varphi(q)$, 
we have
$$
	\sum_{a|q} 
	\sum_{\substack{0\leq b < q/a \\ \gcd(b, \gcd(a, q/a))=1}}
	\widehat{F}_0\left(\bpm \frac{a}{\sqrt{q}} & \frac{b}{\sqrt{q}} \\ & \frac{\sqrt{q}}{a}\ebpm\right)
	c_{q_0}(0) c_{q_1}(0)
	=
	\varphi(q)
	\sum_{a|q} 
	\sum_{\substack{0\leq b < q/a \\ \gcd(b, \gcd(a, q/a))=1}}
	\widehat{F}_0\left(\bpm \frac{a}{\sqrt{q}} & \frac{b}{\sqrt{q}} \\ & \frac{\sqrt{q}}{a}\ebpm\right)
	.
$$
For each $a|q$, 
$$
	\sum_{\substack{0\leq b < q/a \\ \gcd(b, \gcd(a, q/a))=1}}
	\widehat{F}_0\left(\bpm \frac{a}{\sqrt{q}} & \frac{b}{\sqrt{q}} \\ & \frac{\sqrt{q}}{a}\ebpm\right)	
	=
	\sum_{d|\gcd(a, q/a)}
	\mu(d)
	\sum_{0\leq b< \frac{q}{ad}}
	\widehat{F}_0\left(\bpm \frac{a}{\sqrt{q}} & \frac{bd}{\sqrt{q}} \\ & \frac{\sqrt{q}}{a}\ebpm\right).
$$
Take $\ell\in \Z_{\geq 1}$ such that $\ell^2\mid q$ and $q/\ell^2$ is square-free. 
Then we have 
\begin{align*}
	\sum_{a|q} 
	\sum_{\substack{0\leq b < q/a \\ \gcd(b, \gcd(a, q/a))=1}}
	\widehat{F}_0\left(\bpm \frac{a}{\sqrt{q}} & \frac{b}{\sqrt{q}} \\ & \frac{\sqrt{q}}{a}\ebpm\right)
	& = 
	\sum_{a|q} 
	\sum_{d|\gcd(a, q/a)}
	\mu(d)
	\sum_{0\leq b< \frac{q}{ad}}
	\widehat{F}_0\left(\bpm \frac{a}{\sqrt{q}} & \frac{bd}{\sqrt{q}} \\ & \frac{\sqrt{q}}{a}\ebpm\right)
	\\
	& =
	\sum_{d|\ell}
	\mu(d)
	\sum_{a|q/d^2}
	\sum_{0\leq b< \frac{q}{ad^2}}
	\widehat{F}_0\left(\bpm \frac{ad}{\sqrt{q}} & \frac{bd}{\sqrt{q}} \\ & \frac{\sqrt{q}}{ad}\ebpm\right)
	\\ & =
	\sum_{d|\ell}
	\mu(d)
	\sum_{a|q/d^2}
	\sum_{0\leq b< \frac{q}{ad^2}}
	\widehat{F}_0\left(\bpm \frac{a}{\sqrt{q/d^2}} & \frac{b}{\sqrt{q/d^2}} \\ & \frac{\sqrt{q/d^2}}{a}\ebpm\right)
	\\
	& =
	\sum_{d|\ell}
	\mu(d)\sigma_1(q/d^2)
	\left(T_{\frac{q}{d^2}} \widehat{F}_0\right)(1_2).
\end{align*}
Here $T_n$ is the Hecke operator which is defined as in \cite{GM03} by 
$$
	T_n \phi(g) = \frac{1}{\sigma_1(n)} \sum_{\substack{ac=n, \\ 0\leq b < c}} \phi\left(\bpm \frac{a}{\sqrt{n}} & \frac{b}{\sqrt{n}} \\ 0 & \frac{c}{\sqrt{n}}\ebpm g\right), 
$$
for $\phi\in L^2(\SL_2(\Z) \bsl \SL_2(\R))$, for each positive integer $n$. 

Note that since $f$ is bounded, $\widehat{F_0}\in L^2(\SL_2(\Z)\bsl \SL_2(\R))$. 
Following the argument in \cite[\S 3]{GM03}, we get
$$
	\bigg\| T_{n} \widehat{F}_0 - \int_{\SL_2(\Z) \bsl \SL_2(\R)} \widehat{F_0}(g) \; d\mu(g)\bigg\|_2
	\ll 
	n^{-\frac{1}{2}+\epsilon''+\theta} \|\widehat{F_0}\|_{2}, 
$$
for any $\epsilon''>0$. 
Here $\theta$ is the constant towards the Ramanujan conjecture.
We get
\begin{multline}
	\bigg\|
	\frac{\varphi(q)}{\# \scrR_q}
	\sum_{d|\ell}
	\mu(d)\sigma_1(q/d^2)
	T_{\frac{q}{d^2}} \widehat{F}_0
	-
	\int_{\SL_2(\Z) \bsl \SL_2(\R)} 
	\widehat{F}_0(g) \; d\mu(g)
	\bigg\|_2
	\\
	\leq 
	\frac{\varphi(q)}{\# \scrR_q}
	\sum_{d|\ell}
	\left|\mu(d)\right| \sigma_1(q/d^2)
	\bigg\| 
	T_{\frac{q}{d^2}} \widehat{F_0} - 
	\int_{\SL_2(\Z) \bsl \SL_2(\R)} 
	\widehat{F}_0(g) \; d\mu(g)
	\bigg\|_2
	\\
	\ll
	\frac{\varphi(q)}{\# \scrR_q}
	\sum_{d|\ell}
	\sigma_1(q/d^2)
	\left(\frac{q}{d^2}\right)^{-\frac{1}{2}+\epsilon'' +\theta}
	, 
\end{multline}
by the triangular inequality. 
Since
$$
	\sum_{d|\ell}
	\sigma_1(q/d^2)
	\left(\frac{q}{d^2}\right)^{-\frac{1}{2}+\epsilon'' +\theta}
	\ll
	q^{\frac{1}{2}+\epsilon'+\theta}, 
$$
for any $\epsilon' > \epsilon''$, 
we finally get
\begin{equation}
	\bigg\|
	\frac{\varphi(q)}{\# \scrR_q}
	\sum_{d|\ell}
	\mu(d)\sigma_1(q/d^2)
	T_{\frac{q}{d^2}} \widehat{F}_0
	-
	\int_{\SL_2(\Z) \bsl \SL_2(\R)} 
	\widehat{F}_0(g) \; d\mu(g)
	\bigg\|_2
	\ll
	q^{-\frac{1}{2}+\epsilon' +\theta} \big\| \widehat{F}_0 \big\|_2, 
\end{equation}
for any $\epsilon'>0$. 
By Corollary 8.3 in \cite{CU04}, we find that this $L^2$-convergence implies the same rate for point-wise convergence, 
and we are done.  

\section{Proof of Proposition \ref{prop:m=0_error}}  \label{sec:prop4.2}
The standard bound for the Ramanujan sums yields
$$
	c_{q_1}(n_1a+n_2b) c_{q_0} (n_1\overline{b} a+ n_2)
	\leq 
	\gcd(n_1a+n_2b, q_1) \gcd(n_1\overline{b} a+n_2, q_0)
	.
$$
When $n_2=0$, then $n_1\neq 0$, and 
$$
	\gcd(n_1a, q_1) \gcd(n_1\overline{b} a, q_0) 
	=
	\gcd(n_1, q_1) a \gcd(n_1, q_0/a)
	\leq 
	aq^{2\epsilon}.
$$
for $0\neq |n_1|<q^{\epsilon}$. 
In this case, we have
\begin{multline}
	\bigg|
	\sum_{a|q} 
	\sum_{\substack{0\leq b < q/a \\ \gcd(b, \gcd(a, q/a))=1}}
	\widehat{F}_0\left(\bpm \frac{a}{\sqrt{q}} & \frac{b}{\sqrt{q}} \\ & \frac{\sqrt{q}}{a}\ebpm\right)
	c_{q_1}(n_1ab) c_{q_0} (n_1\overline{b} a)
	\bigg| \\
	\ll_F
	\sum_{a|q} 
	\sum_{\substack{0\leq b < q/a \\ \gcd(b, \gcd(a, q/a))=1}}
	aq^{2\epsilon}
	\leq
	\sum_{a|q}
	q^{1+2\epsilon}
	=
	q^{1+2\epsilon}
	\sigma_0(q) 
	.
\end{multline}

Assume that $n_2\neq 0$. 
Then
\begin{multline}\label{e:not_00}
	\gcd(n_1\overline{b}a+n_2, q_0)
	\leq 
	\gcd(n_1\overline{b}a+n_2, a)
	\gcd(n_1\overline{b}a+n_2, q_0/a)
	\\
	=
	\gcd(n_2, a)
	\gcd(n_1a+n_2b, q_0/a)
	\leq 
	q^\epsilon
	\gcd(n_1a+n_2b, q_0/a), 
\end{multline}
for $0\neq |n_2| < q^{\epsilon}$. 
We get
\begin{multline*}
	\bigg|
	\sum_{a|q} 
	\sum_{\substack{0\leq b < q/a \\ \gcd(b, \gcd(a, q/a))=1}}
	\widehat{F}_0\left(\bpm \frac{a}{\sqrt{q}} & \frac{b}{\sqrt{q}} \\ & \frac{\sqrt{q}}{a}\ebpm\right)
	c_{q_1}(n_1a+n_2b) c_{q_0} (n_1\overline{b} a+ n_2)
	\bigg|
	\\
	\ll_F
	q^{\epsilon}
	\sum_{a|q} 
	\sum_{\substack{0\leq b < q/a \\ \gcd(b, \gcd(a, q/a))=1}}
	\gcd(n_1a+n_2b, q/a).
\end{multline*}
Now
$$
\gcd(n_1a+n_2b, q/a) =\sum_{\substack{k|\frac{q}{a}\\ k | n_1a+n_2b}} \varphi(k)
\leq \sum_{\substack{k|\frac{q}{a}\\ k | n_1a+n_2b}} k.
$$
For each $k\mid \frac{q}{a}$, we have
\begin{align}\label{e:gcd_b_number}
	&\# \left\{ 0\leq b < q/a\;:\; n_2b\equiv -n_1a\pmod{k}\right\} \\
	& \leq\# \left\{0 \leq x < |n_2| \tfrac{q}{a}\;:\; x\equiv \mp n_1 a\pmod{k}\right\} \notag \\
	& = \# \left( \Z\cap \bigg[\pm \frac{n_1 a}{k} , \pm \frac{n_1 a}{k}  + |n_2| \frac{q/a}{k } \bigg) \right) \notag \\
	 &= |n_2| \frac{q/a}{k} \notag, 
\end{align}
where $\mp$ is chosen according to the sign of $n_2$. 
Then
\begin{multline*}
	\sum_{a|q} 
	\sum_{\substack{0\leq b < q/a \\ \gcd(b, \gcd(a, q/a))=1}}
	\gcd(n_1a+n_2b, q/a)
	\leq 
	\sum_{a\mid q} \sum_{k\mid \frac{q}{a}}
	|n_2| \frac{q}{a} \\
	< 
	q^{\epsilon} \sum_{a\mid q} \frac{q}{a} \sigma_0(q/a) 
	=
	q^{\epsilon}\sum_{a\mid q} a\sigma_0(a)  
	\leq 
	q^{\epsilon} \sigma_0(q) \sigma_1(q) \leq q^{1+\epsilon} \sigma_0(q)^2. 
\end{multline*}
Therefore, we have
\begin{equation}
	\bigg|
	\sum_{a|q} 
	\sum_{\substack{0\leq b < q/a \\ \gcd(b, \gcd(a, q/a))=1}}
	\widehat{F}_0\left(\bpm \frac{a}{\sqrt{q}} & \frac{b}{\sqrt{q}} \\ & \frac{\sqrt{q}}{a}\ebpm\right)
	c_{q_1}(n_1a+n_2b) c_{q_0} (n_1\overline{b} a+ n_2)
	\bigg|
	\ll_F
	q^{1+2\epsilon} \sigma_0(q)^2 
\end{equation}
and the claim follows via the lower bound \eqref{lowerR_q}.
\qed

\section{Proof of Proposition \ref{prop:mneq0}}  \label{sec:prop4.3}
For $\epsilon>0$, consider 
\begin{equation}\label{e:cbd_q}
	\max\left\{ m|c|\frac{a}{\sqrt{q}}, \; m\frac{\sqrt{q}}{a}\left|\frac{bc}{q/a} + d\right|\right\} \leq q^\epsilon.
\end{equation}

When $c=0$, then $d=\pm1$, and we have 
\begin{equation}\label{e:bcd_q_c=0}
	m\frac{\sqrt{q}}{a} |d| = m\frac{\sqrt{q}}{a} \leq q^\epsilon. 
\end{equation}
It follows that $a\geq q^{\frac{1}{2}-\epsilon}$ because 
$m\geq 1$ and $\frac{\sqrt{q}}{a} \leq q^\epsilon$. 
Also, we get $m \leq a q^{-\frac{1}{2}+\epsilon}$ from \eqref{e:bcd_q_c=0}. 
For each $a\mid q$ and $a\geq q^{\frac{1}{2}-\epsilon}$, let $P_1(a)$ be the $(c=0)$-part of the sum appearing in Proposition \ref{prop:mneq0}: 
\begin{multline*}
	P_1(a)
	:=
	\sum_{\substack{0\leq b < q/a \\ \gcd(b, \gcd(a, q/a))=1}}
	\sum_{1\leq m < aq^{-\frac{1}{2}+\epsilon}}
	\sum_{d=\pm 1}
	\widehat{F}_m \left(\bpm * & * \\ 0& d\ebpm \bpm \frac{a}{\sqrt{q}} & \frac{b}{\sqrt{q}} \\ & \frac{\sqrt{q}}{a}\ebpm\right)
	\\
	\times
	c_{q_1}(n_1a+n_2b; q_1)
	S\left( (n_1\overline{b} a+ n_2) \ell_1, d\tfrac{q_0}{a}m ; q_0 \right)
	.
\end{multline*}
Then 
\begin{equation}\label{e:c=0}
	\left|P_1(a)\right|
	\ll_F
	\sum_{\substack{0\leq b< q/a \\ \gcd(b, \gcd(a, q/a))=1}}
	\sum_{1\leq m < aq^{-\frac{1}{2}+\epsilon}} 
	\sum_{d=\pm 1} 
	\left|c_{q_1}(n_1a+n_2b; q_1)
	S\left( (n_1\overline{b} a+ n_2) \ell_1, d\tfrac{q_0}{a}m ; q_0 \right)\right|.
\end{equation}

When $c\neq 0$, we have 
\begin{equation}\label{e:bcd_cneq0}
	m|c| \frac{a}{\sqrt{q}} \leq q^\epsilon \text{ and } 
	m\frac{\sqrt{q}}{a} \left|\frac{bc}{q/a}+d\right| \leq q^\epsilon. 
\end{equation}
Since $m\geq 1$ and $1\leq |c|$, we have $\frac{a}{\sqrt{q}} < q^\epsilon$, so $a<q^{\frac{1}{2}+\epsilon}$. 
Moreover, $m|c| < \frac{q^{\frac{1}{2}+\epsilon}}{a}$. 
For each $a\mid q$ and $a\leq q^{\frac{1}{2}+\epsilon}$, 
let $P_2(a)$ be the $(c\neq 0)$-part of the sum appearing in Proposition \ref{prop:mneq0}: 
\begin{multline}
	P_2(a)
	:=
	\sum_{1\leq m<\frac{q^{\frac{1}{2}+\epsilon}}{a}}
	\sum_{\substack{c\in \Z \\ 1\leq |c| \leq \frac{q^{\frac{1}{2}+\epsilon}}{am}}}
	\sum_{\substack{0\leq b< q/a \\ \gcd(b, \gcd(a, q/a))=1}}
	\sum_{\substack{d\in \Z \\ \gcd(c, d)=1 \\ \left|\frac{bc}{q/a}+d\right| \leq T=\frac{aq^{-\frac{1}{2}+\epsilon}}{m}}}
	\widehat{F}_m \left(\bpm * & * \\ c& d\ebpm \bpm \frac{a}{\sqrt{q}} & \frac{b}{\sqrt{q}} \\ & \frac{\sqrt{q}}{a}\ebpm\right)
	\\
	\times
	S\left( (n_1a+n_2b) \ell_0, cm\ell_0; q_1\right) 
	S\left( (n_1\overline{b} a+ n_2) \ell_1, (cb\ell_1 + d\tfrac{q_0}{a})m ; q_0 \right)
	.
\end{multline}
Then 
\begin{multline}\label{e:P2}
	\left|P_2(a)\right|
	\ll_F
	\sum_{1\leq m<\frac{q^{\frac{1}{2}+\epsilon}}{a}}
	\sum_{\substack{c\in \Z \\ 1\leq |c| \leq \frac{q^{\frac{1}{2}+\epsilon}}{am}}}
	\sum_{\substack{0\leq b< q/a \\ \gcd(b, \gcd(a, q/a))=1}}
	\sum_{\substack{d\in \Z \\ \gcd(c, d)=1 \\ \left|\frac{bc}{q/a}+d\right| \leq \frac{aq^{-\frac{1}{2}+\epsilon}}{m}}}
	\\
	\times
	\left|S\left( (n_1a+n_2b) \ell_0, cm\ell_0; q_1\right) 
	S\left( (n_1\overline{b} a+ n_2) \ell_1, (cb\ell_1 + d\tfrac{q_0}{a})m ; q_0 \right)\right|
	.
\end{multline}

Proposition \ref{prop:mneq0} follows from the next two lemmas and \eqref{lowerR_q}.
\begin{lemma}\label{lem:P1}
For $a\mid q$ and $a>q^{\frac{1}{2}-\epsilon}$, 
\begin{equation}\label{e:P1_bound}
	\left|P_1(a)\right| \ll_F \sigma_0(q)^2 q^{\frac{5}{4}+\frac{5}{2}\epsilon} . 
\end{equation}
\end{lemma}
\begin{proof}

%
%
%

When $n_1=n_2=0$, 
$$
	\left|\varphi(q_1)
	c_{q_0}\left(\tfrac{q_0}{a}m\right)\right|
	\leq 
	\varphi(q_1) 
	\frac{q_0}{a}
	\gcd(m, a)
	\leq
	\frac{q}{a}
	\gcd(m, a)
	.
$$
Applying to \eqref{e:c=0} yields
$$
	\left|P_1(a)\right|
	\ll_F
	\frac{q}{a}
	\sum_{\substack{0 \leq b< q/a, \\ \gcd(b, \gcd(a, q/a))=1}}
	\sum_{1\leq m < aq^{-\frac{1}{2}+\epsilon}} 
	\gcd(m, a)
	.
$$
Since 
\begin{multline}\label{e:P1_middle}
	\sum_{\substack{0 \leq b< q/a, \\ \gcd(b, \gcd(a, q/a))=1}}
	\sum_{1\leq m < aq^{-\frac{1}{2}+\epsilon}} 
	\gcd(m, a)
	\leq 
	\sum_{\substack{0 \leq b< q/a, \\ \gcd(b, \gcd(a, q/a))=1}}
	\sum_{\substack{k\mid a, \\ 1\leq mk < aq^{-\frac{1}{2}+\epsilon}} }
	k
	\leq
	q^{\frac{1}{2}+\epsilon}
	\sigma_0(a)
	, 
\end{multline}
we find
$$
	\left|P_1(a)\right|
	\ll_F
	\frac{q^{\frac{3}{2}+\epsilon}}{a} \sigma_0(q)  
	.
$$
Since $a>q^{\frac{1}{2}-\epsilon}$, 
\begin{equation}
	\left|P_1(a)\right| 
	\ll_F 
	q^{1+2\epsilon} \sigma_0(a). 
\end{equation}

Assume now $(n_1, n_2)\neq (0, 0)$. 
Recalling \eqref{e:c=0}, applying the Weil's bound of the Kloosterman sum and the well-known bound of the Ramanujan sum, 
we get
\begin{align}
	\bigg|c_{q_1}(n_1a+n_2b) & S\left( (n_1\overline{b} a+ n_2) \ell_1, d\tfrac{q_0}{a}m ; q_0 \right)\bigg|
	\nonumber
	\\
	& \leq 
	\gcd(n_1a+n_2b, q_1)
	\sigma_0(q_0) \sqrt{q_0} 
	\gcd(n_1\overline{b} a+n_2, \tfrac{q_0}{a} m, q_0)^{\frac{1}{2}}
	\label{e:c=0_Ram_Klo}
	\\
	& \leq
	\sqrt{q}\sigma_0(q_0)
	\gcd(n_1a+n_2b, q_1)^{\frac{1}{2}}
	\gcd(n_1\overline{b} a+n_2, q_0)^{\frac{1}{2}}
	\label{e:c=0_Ram_Klo2}
\end{align}
Note that
\begin{equation}\label{e:c=0_gcd}
	\gcd(n_1\overline{b} a+n_2, \tfrac{q_0}{a} m, q_0)
	\leq 
	\frac{q_0}{a}\gcd(m, a).
\end{equation}

Consider first that $n_1\neq 0$ and $n_2=0$. 
Since $\gcd(a, q_1)=1$ and $0\neq |n_1| < q^\epsilon$, 
$$
	\gcd(n_1a+n_2b, q_1) = \gcd(n_1a, q_1) = q^\epsilon. 
$$
Then by \eqref{e:c=0_Ram_Klo}, 
$$
	\left|c_{q_1}(n_1a; q_1)
	S\left( n_1\overline{b} a \ell_1, \tfrac{q_0}{a}m ; q_0 \right)\right|
	\leq 
	q^\epsilon
	\sigma_0(q_0) 
	\frac{q_0}{\sqrt{a}}
	\gcd(m, a)^{\frac{1}{2}}
	.
$$
Recalling \eqref{e:c=0}, we deduce
$$
	\left|P_1(a)\right|
	\ll_F
	q^\epsilon
	\sigma_0(q_0) 
	\frac{q_0}{\sqrt{a}}
	\sum_{\substack{0 \leq b< q/a, \\ \gcd(b, \gcd(a, q/a))=1}}
	\sum_{1\leq m < aq^{-\frac{1}{2}+\epsilon}} 
	\gcd(m, a)^{\frac{1}{2}}
$$
Similar to \eqref{e:P1_middle}, and since $a>q^{\frac{1}{2}-\epsilon}$, we have
\begin{equation}
	\left|P_1(a)\right|
	\ll_F	
	q^{\frac{1}{2}+2\epsilon}
	\sigma_0(q_0) \sigma_{-1/2}(a)
	\frac{q_0}{\sqrt{a}}
	\leq 
	q^{\frac{5}{4}+\frac{5\epsilon}{2}} \sigma_0(q)^2.
\end{equation}

When $n_2\neq 0$, in view of \eqref{e:not_00}, 
$$
	\gcd(n_1a+n_2b, q_1)
	\gcd(n_1\overline{b}a+n_2, q_0)
	\leq 
	q^\epsilon
	\gcd(n_1a+n_2b, q/a). 
$$
In view of \eqref{e:c=0_Ram_Klo2}, we obtain
\begin{equation*}
	\left|c_{q_1}(n_1a+n_2b)
	S\left( (n_1\overline{b} a+ n_2) \ell_1, \tfrac{q_0}{a}m ; q_0 \right)\right|
	\leq
	\sigma_0(q_0) q^{\frac{1}{2}+\epsilon} \gcd(n_1a+n_2b, q/a)^{\frac{1}{2}}. 
\end{equation*}
Recalling \eqref{e:c=0}, 
$$
	\left|P_1(a)\right|
	\ll_F
	\sigma_0(q_0) aq^{2\epsilon} 
	\sum_{\substack{0\leq b< q/a \\ \gcd(b, \gcd(a, q/a))=1}}
	\gcd(n_1a+n_2b, q/a)^{\frac{1}{2}}
	.
$$
By \eqref{e:gcd_b_number}, 
$$
	\sum_{\substack{0\leq b< q/a \\ \gcd(b, \gcd(a, q/a))=1}}
	\gcd(n_1a+n_2b, q/a)^{\frac{1}{2}}
	\leq
	\sum_{k\mid \frac{q}{a}}
	k^{\frac{1}{2}}
	\sum_{\substack{0\leq b< q/a \\ n_2b\equiv -n_1a\pmod{k}}}
	1
	\leq 
	\frac{q^{1+\epsilon}}{a}
	\sum_{k\mid \frac{q}{a}}
	k^{-\frac{1}{2}}
	.
$$
So 
\begin{equation}
	\left|P_1(a)\right| \ll_F \sigma_0(q)^2 q^{1+3\epsilon}.
\end{equation}
This concludes the proof.
\end{proof}

\begin{lemma}\label{lem:P2}
For each $a\mid q$ and $a\leq q^{\frac{1}{2}+\epsilon}$, 
\begin{equation}\label{e:P2_bound}
	\left|P_2(a)\right|
	\ll_F
	\sigma_0(q)^2 
	q^{\frac{3}{2}+3\epsilon}
	.
\end{equation}
\end{lemma}
\begin{proof}
From the second inequality in \eqref{e:bcd_cneq0}, 
\begin{equation}
	\left|\frac{bc}{q/a}+d\right| \leq \frac{aq^{-\frac{1}{2}+\epsilon}}{m}=:T. 
\end{equation}
Then one can separate $\left|P_2(a)\right|$ into two parts. 
For $T\geq \frac{1}{2}$, let
\begin{multline}\label{e:P3}
	P_3(a)
	:=
	\sum_{\substack{1\leq m<\frac{q^{\frac{1}{2}+\epsilon}}{a}, 
	\\ T\geq \frac{1}{2}}}
	\sum_{\substack{c\in \Z \\ 1\leq |c| < \frac{q^{\frac{1}{2}+\epsilon}}{am}}}
	\sum_{\substack{0\leq b< q/a \\ \gcd(b, \gcd(a, q/a))=1}}
	\sum_{\substack{d\in \Z \\ \gcd(c, d)=1 \\ \left|\frac{bc}{q/a}+d\right| \leq T}}
	\\
	\times
	\left|S\left( (n_1a+n_2b) \ell_0, cm\ell_0; q_1\right) 
	S\left( (n_1\overline{b} a+ n_2) \ell_1, (cb\ell_1 + d\tfrac{q_0}{a})m ; q_0 \right)\right|
	, 
\end{multline}
and for $T< \frac{1}{2}$, let
\begin{multline}\label{e:P4}
	P_4(a) 
	:=
	\sum_{\substack{1\leq m<\frac{q^{\frac{1}{2}+\epsilon}}{a} \\ T< \frac{1}{2}}}
	\sum_{\substack{c\in \Z \\ 1\leq |c| < \frac{q^{\frac{1}{2}+\epsilon}}{am}}}
	\sum_{\substack{0\leq b< q/a \\ \gcd(b, \gcd(a, q/a))=1}}
	\sum_{\substack{d\in \Z \\ \gcd(c, d)=1 \\ \left|\frac{bc}{q/a}+d\right| \leq T}}
	\\
	\times
	\left|S\left( (n_1a+n_2b) \ell_0, cm\ell_0; q_1\right) 
	S\left( (n_1\overline{b} a+ n_2) \ell_1, (cb\ell_1 + d\tfrac{q_0}{a})m ; q_0 \right)\right|.
\end{multline}
Then $\left|P_2(a)\right| \ll_F P_3(a) +P_4(a)$. 

In the case $T\geq \frac{1}{2}$,  
\begin{equation}\label{e:d_T>1/2}
	\#\left\{ d\in \Z\;:\; \left|\frac{bc}{q/a} + d\right| \leq T\right\} \leq 2T+1 \leq 4T. 
\end{equation}
Note that since $T=\frac{aq^{-\frac{1}{2}+\epsilon}}{m} \geq \frac{1}{2}$ and $m\geq 1$, 
we have $a\geq \frac{1}{2}q^{\frac{1}{2}-\epsilon}$. 
By applying the trivial bound of the Kloosterman sums to \eqref{e:P3}, 
we obtain
$$
	P_3(a) 
	\leq 
	\varphi(q)
	\sum_{1\leq m<\frac{q^{\frac{1}{2}+\epsilon}}{a}}
	\sum_{\substack{c\in \Z \\ 1\leq |c| < \frac{q^{\frac{1}{2}+\epsilon}}{am}}}
	\sum_{\substack{0\leq b< q/a \\ \gcd(b, \gcd(a, q/a))=1}}
	\sum_{\substack{d\in \Z \\ \gcd(c, d)=1 \\ \left|\frac{bc}{q/a}+d\right| \leq T}}
	1.
$$
By \eqref{e:d_T>1/2},  
$$
	\sum_{\substack{c\in \Z \\ 1\leq |c| < \frac{q^{\frac{1}{2}+\epsilon}}{am}}}
	\sum_{\substack{0\leq b< q/a \\ \gcd(b, \gcd(a, q/a))=1}}
	\sum_{\substack{d\in \Z \\ \gcd(c, d)=1 \\ \left|\frac{bc}{q/a}+d\right| \leq T}} 1
	\leq 
	4\frac{q}{a}
	\sum_{\substack{c\in \Z \\ 1\leq |c| < \frac{q^{\frac{1}{2}+\epsilon}}{am}}} 
	T
	\leq 
	8\frac{q}{a} \frac{q^{\frac{1}{2}+\epsilon}}{am} T
	=
	8\frac{q^{1+2\epsilon}}{a} \frac{1}{m^2}
	.
$$
So we have
$$
	P_3(a) 
	\leq 
	8\varphi(q)\frac{q^{1+2\epsilon}}{a}
	\sum_{m=1}^\infty \frac{1}{m^2}
	=
	8 \zeta(2) \varphi(q) \frac{q^{1+2\epsilon}}{a}. 
$$
Because $a\geq \frac{1}{2}q^{\frac{1}{2}-\epsilon}$, we have
\begin{equation}\label{e:P3_fin}
	P_3(a) 
	\leq 
	16 \zeta(2) \varphi(q) q^{\frac{1}{2}+3\epsilon}.
\end{equation}


For $T< \frac{1}{2}$, recall \eqref{e:P4}. 
Applying Weil's bound of the Kloosterman sums, 
$$
	\left|S( (n_1a+n_2b)\ell_0, cm\ell_0; q_1) \right|
	\leq 
	\sigma_0(q_1)
	\sqrt{q_1}
	\gcd( (n_1a+n_2b)\ell_0, cm\ell_0, q_1)^{\frac{1}{2}}
	.
$$
Since $\gcd(\ell_0, q_1)=1$, 
$$
	\gcd( (n_1a+n_2b)\ell_0, cm\ell_0, q_1)
	=
	\gcd( (n_1a+n_2b), cm, q_1)
	\leq 
	\gcd(cm, q_1). 
$$
So 
\begin{equation}\label{e:K_S1}
	\left|S( (n_1a+n_2b)\ell_0, cm\ell_0; q_1) \right|
	\leq 
	\sigma_0(q_1)
	\sqrt{q_1}
	\gcd(cm, q_1)^{\frac{1}{2}}
	.
\end{equation}
Similarly, for the second Kloosterman sum, 
\begin{multline*}
	\left|S( (n_1\bar{b} a+ n_2)\ell_1, m(cb\ell_1+d\tfrac{q_0}{a}); q_0)\right|
	\\
	\leq
	\sigma_0(q_0)
	\sqrt{q_0}
	\gcd((n_1\bar{b} a+ n_2)\ell_1, m(cb\ell_1+d\tfrac{q_0}{a}), q_0)^{\frac{1}{2}}
	\\
	\leq 
	\sigma_0(q_0)
	\sqrt{q_0}
	\gcd(m(cb\ell_1+d\tfrac{q_0}{a}), q_0)^{\frac{1}{2}}
	.
\end{multline*}
Since $\gcd(b\ell_1, q_0/a)=1$, 
$$
	\gcd(m(cb\ell_1+d\tfrac{q_0}{a}), q_0)
	\leq 
	a \cdot \gcd(m(cb\ell_1+d\tfrac{q_0}{a}), q_0/a)
	\leq 
	a \gcd(mc, q_0/a).
$$
So  
\begin{equation}\label{e:K_S2}
	\left|S( (n_1\bar{b} a+ n_2)\ell_1, m(cb\ell_1+d\tfrac{q_0}{a}); q_0)\right|
	\leq
	\sigma_0(q_0)
	\sqrt{aq_0}
	\gcd(mc, q_0/a)^{\frac{1}{2}}
	.
\end{equation}
Combining \eqref{e:K_S1} and \eqref{e:K_S2}, 
for $\gcd(q_1, q_0/a)=1$, we get
\begin{multline}\label{e:K_bound}
	\left|S\left( (n_1a+n_2b) \ell_0, cm\ell_0; q_1\right) 
	S\left( (n_1\overline{b} a+ n_2) \ell_1, (cb\ell_1 + d\tfrac{q_0}{a})m ; q_0 \right)\right|
	\\
	\leq
	\sigma_0(q)
	\sqrt{q}\sqrt{a}
	\gcd(cm, q/a)^{\frac{1}{2}}
	.
\end{multline}
Applying \eqref{e:K_bound} to \eqref{e:P4}, 
\begin{equation}\label{e:P4_starting}
	P_4(a)
	\leq 
	\sigma_0(q)
	\sqrt{qa}
	\sum_{\substack{1\leq m<\frac{q^{\frac{1}{2}+\epsilon}}{a}, \\ T< \frac{1}{2}} }
	\sum_{\substack{c\in \Z \\ 1\leq |c| < \frac{q^{\frac{1}{2}+\epsilon}}{am}}}
	\gcd(cm, q/a)^{\frac{1}{2}}
	\sum_{\substack{0\leq b< q/a \\ \gcd(b, \gcd(a, q/a))=1}}
	\sum_{\substack{d\in \Z \\ \gcd(c, d)=1 \\ \left|\frac{bc}{q/a}+d\right| \leq T}}1 .
\end{equation}	

When $T< \frac{1}{2}$, for given $0\leq b< q/a$ and $c\neq 0$, 
there exist uniquely determined $d\in \Z$, such that 
$$
	\left|\frac{bc}{q/a}+d\right|
	\leq 
	T=\frac{aq^{-\frac{1}{2}+\epsilon}}{m}. 
$$
Set $\alpha=bc+d\frac{q}{a}\in \Z$. 

Conversely, 
for $\alpha\in \Z$ with $\left|\alpha\right| \leq T\frac{q}{a}< \frac{1}{2}\frac{q}{a}$ 
and $c\neq 0$, 
we take $0\leq b <q/a$ as the solution of the congruence equation: 
\begin{equation}\label{e:bcalpha}
	bc \equiv \alpha \pmod{q/a}. 
\end{equation}
Then $d\in \Z$ is uniquely determined by 
$$
	\frac{bc-\alpha}{q/a}=d. 
$$
Therefore for a given integer $c\neq 0$, 
\begin{equation}\label{e:T<1/2_bcounting}
	\sum_{0\leq b< q/a}
	\sum_{\substack{d\in \Z \\ \gcd(c, d)=1 \\ \left|\frac{bc}{q/a}+d\right| \leq T}}
	1
	=
	\sum_{\substack{\alpha\in \Z, \\ |\alpha| \leq T\frac{q}{a}}}
	\sum_{\substack{0\leq b <q/a, \\ bc\equiv \alpha\pmod{q/a}}}
	1.
\end{equation}

Let $c_1:=\gcd(c, q/a)$. If there exists a solution $0\leq b< q/a$ of \eqref{e:bcalpha}, 
$c_1$ must divide $\alpha$. 
Take $\tilde{c}\in \Z$ as $\tilde{c} \frac{c}{c_1} \equiv 1\pmod{\frac{q}{ac_1}}$. 
Then 
$$
	b\equiv \frac{\alpha}{c_1} \tilde{c} \pmod{\frac{q}{ac_1}}. 
$$
So for $c\neq 0$ and $\alpha\in \Z$, if $\gcd(c, q/a) \mid \alpha$, 
\begin{align}\label{e:number_bcalpha}
\begin{split}
	\sum_{\substack{0\leq b <q/a, \\ bc\equiv \alpha\pmod{q/a}}}
	1
	&=
	\# \left\{0\leq b < q/a\;:\; bc\equiv \alpha\pmod{q/a}\right\}
	\\
	&=
 	\left\{ b=\tfrac{\alpha}{c_1} \tilde{c} +j \tfrac{q}{ac_1}\;:\; 0\leq j < c_1\right\}
	=
	\gcd(c, q/a).
\end{split}
\end{align}
If $\gcd(c, q/a)\nmid \alpha$, then 
$$
	\sum_{\substack{0\leq b <q/a, \\ bc\equiv \alpha\pmod{q/a}}}
	1
	=0. 
$$
Therefore, for each given $c\neq 0$, 
\begin{align}
\begin{split}
	\sum_{\substack{\alpha\in \Z, \\ |\alpha| \leq T\frac{q}{a}}}
	\sum_{\substack{0\leq b <q/a, \\ bc\equiv \alpha\pmod{q/a}}}
	1
	& =
	\gcd(c, q/a)
	\sum_{\substack{\alpha\in \Z, \; \gcd(c,q/a) \mid \alpha \\ |\alpha| \leq T\frac{q}{a}}}
	1
	\\
	& \leq 
	\left(2T\frac{q}{a} \frac{1}{\gcd(c, q/a)}+1\right) \gcd(c, q/a) = 2T\frac{q}{a}+ \gcd(c, q/a). 
\end{split}
\end{align}

Applying this to \eqref{e:T<1/2_bcounting} yields
$$
	\sum_{\substack{0\leq b< q/a \\ \gcd(b, \gcd(a, q/a))=1}}
	\sum_{\substack{d\in \Z \\ \gcd(c, d)=1 \\ \left|\frac{bc}{q/a}+d\right| \leq T}}
	1
	\leq 
	2T\frac{q}{a}+\gcd(c, q/a) .
$$
Recalling \eqref{e:P4_starting}, we finally get
\begin{equation}
	P_4(a)
	\leq 
	\sigma_0(q)
	\sqrt{qa}
	\sum_{1\leq m<\frac{q^{\frac{1}{2}+\epsilon}}{a}}
	\sum_{\substack{c\in \Z \\ 1\leq |c| < \frac{q^{\frac{1}{2}+\epsilon}}{am}}}
	\gcd(cm, q/a)^{\frac{1}{2}}
	\left(2T\frac{q}{a}+\gcd(c, q/a)\right).
\end{equation}
Since $\gcd(c, q/a)\leq |c| < \frac{q^{\frac{1}{2}+\epsilon}}{am} < \frac{q^{\frac{1}{2}+\epsilon}}{m}$ 
and 
$T=\frac{aq^{-\frac{1}{2}+\epsilon}}{m}$, 
we get
$$
	\gcd(c, q/a) < \frac{q^{\frac{1}{2}+\epsilon}}{m} = T\frac{q}{a}
	.
$$
Then 
\begin{align}
	P_4(a) 
	&\leq 
	\sigma_0(q)
	\sqrt{qa}
	\sum_{1\leq m<\frac{q^{\frac{1}{2}+\epsilon}}{a}}
	\sum_{\substack{c\in \Z \\ 1\leq |c| < \frac{q^{\frac{1}{2}+\epsilon}}{am}}}
	\gcd(cm, q/a)^{\frac{1}{2}}
	3T\frac{q}{a}
	\nonumber
	\\
	&\leq 
	3
	\sigma_0(q)
	q^{1+\epsilon} \sqrt{a}
	\sum_{1\leq m<\frac{q^{\frac{1}{2}+\epsilon}}{a}}
	\sum_{\substack{c\in \Z \\ 1\leq |c|m < \frac{q^{\frac{1}{2}+\epsilon}}{a}}}
	\gcd(cm, q/a)^{\frac{1}{2}}
	\frac{1}{m}
	\label{e:P4_mid}
\end{align}	
Let $k:=\gcd(cm, q/a)$. 
Then 
\begin{multline*}
	\sum_{1\leq m<\frac{q^{\frac{1}{2}+\epsilon}}{a}}
	\sum_{\substack{c\in \Z, c\neq 0 \\ 1\leq |c|m < \frac{q^{\frac{1}{2}+\epsilon}}{a}}}
	\gcd(cm, q/a)^{\frac{1}{2}}
	\frac{1}{m}
	\leq 
	\sum_{k\mid q/a}
	\sum_{1\leq m<\frac{q^{\frac{1}{2}+\epsilon}}{a}}
	\sum_{\substack{c\in \Z, c\neq 0 \\ 1\leq |c|mk < \frac{q^{\frac{1}{2}+\epsilon}}{a}}}
	\frac{\sqrt{k}}{m}
	\\
	\leq 
	2
	\sum_{k\mid q/a}
	\frac{1}{\sqrt{k}}
	\sum_{1\leq m<\frac{q^{\frac{1}{2}+\epsilon}}{a}}
	\frac{q^{\frac{1}{2}+\epsilon}}{am^2}
	\leq
	2
	\sigma_0(q/a) \frac{q^{\frac{1}{2}+\epsilon}}{a}
	\zeta(2).
\end{multline*}
Applying this to \eqref{e:P4_mid} yields
\begin{equation}\label{e:P4_fin}
	P_4(a)
	\leq 
	6\zeta(2)
	\sigma_0(q)\sigma_0(q/a) 
	\frac{q^{\frac{3}{2}+2\epsilon} }{\sqrt{a}}.
\end{equation}	

Finally, by \eqref{e:P3_fin} and \eqref{e:P4_fin}, 
\begin{equation}
	\left|P_2(a)\right|
	\ll_F 
	16 \zeta(2) \varphi(q) q^{\frac{1}{2}+3\epsilon}
	+6\zeta(2)
	\sigma_0(q)\sigma_0(q/a) 
	\frac{q^{\frac{3}{2}+2\epsilon} }{\sqrt{a}} \ll \sigma_0(q)^2 q^{\frac{3}{2}+3\epsilon} .
\end{equation}
\end{proof}

\thispagestyle{empty}
{\footnotesize
\nocite{*}
\bibliographystyle{amsalpha}
\bibliography{reference_EEH}
}
\end{document}